\documentclass[11pt]{article}
\usepackage[utf8]{inputenc}

\usepackage[paperwidth=8.5in, 
		paperheight=11in, 
		portrait, 
		top=1in, 
		bottom=1.in, 
		left=1.15in, 
		right=1.15in
		]{geometry}

\usepackage{authblk}
\usepackage{blkarray}
\usepackage{amsmath}

\usepackage{amsfonts,amsmath,amssymb,amsthm}
\usepackage{latexsym,mathrsfs,mathtools,bm}
\usepackage{graphicx,xcolor}
\usepackage{braket}
\usepackage{enumitem}
\usepackage{chngcntr}

\usepackage{tikz}
\usetikzlibrary{calc}

\usepackage[hidelinks]{hyperref}
\usepackage{bookmark}

\def\cA{{\mathcal A}}   \def\cB{{\mathcal B}}   \def\cC{{\mathcal C}}
\def\cD{{\mathcal D}}
\def\cU{{\mathcal U}}
\def\cV{{\mathcal V}}
\def\cW{{\mathcal W}}

\newcommand{\qhg}[2]{\,\mbox{}_{#1}\phi_{ #2}\!}
\newcommand{\qargu}[4]{\left(\begin{array}{c} #1\\#2\end{array} ; #3, #4\right)}

\theoremstyle{plain}
\newtheorem{thm}{Theorem}[section]

\newtheorem{prop}[thm]{Proposition}

\newtheorem{defi}[thm]{Definition}
\newtheorem{rem}[thm]{Remark}

\numberwithin{equation}{section}

\newcommand{\mqh}{m\mathfrak{h}_q}
\newcommand{\qh}{\mathfrak{h}_q}
\newcommand{\V}{\mathcal{V}_N}

\newcommand{\Q}[1]{[#1]_q}

\newcommand{\Xt}{\widetilde{X}}
\newcommand{\Vt}{\widetilde{V}}
\newcommand{\Zt}{\widetilde{Z}}

\title{\bf Meta algebras and biorthogonal rational functions:\\ the $q$-Hahn case}

\renewcommand*{\Affilfont}{\normalsize\small}
\author[1]{Pierre-Antoine Bernard}
\author[2]{Abderahmane Bouziane}
\author[3]{Samuel Pellerin}
\author[4]{Simone Têtu}
\author[5]{Satoshi Tsujimoto}
\author[6]{Luc Vinet}
\author[7]{Meri Zaimi}
\author[8]{Alexei Zhedanov\vspace{.5em}}
\affil[1,2,3,6]{\textit{Centre de Recherches Math\'ematiques, Universit\'e de Montr\'eal, P.O. Box 6128, Centre-ville Station, Montr\'eal (Qu\'ebec), H3C 3J7, Canada. \vspace{.9em}}}
\affil[4]{\textit{Division of Applied Mathematics, Brown University, Providence RI 02912, United States. \vspace{.9em}}}
\affil[5]{\textit{Graduate School of Informatics, Kyoto University, Yoshida-Honmachi, Kyoto, Japan 606-8501.}
\vspace{.9em}}
\affil[6]{\textit{IVADO, Montr\'eal (Qu\'ebec), H2S 3H1, Canada.} \vspace{.9em}}
\affil[7]{\textit{Perimeter Institute for Theoretical Physics, Waterloo (Ontario), N2L 2Y5, Canada.} \vspace{.9em}}
\affil[8]{\textit{School of Mathematics, Renmin University of China, Beijing, 100872, China.} \vspace{.9em}}

{
\makeatletter
\renewcommand\AB@affilsepx{: \protect\Affilfont}
\makeatother
\affil[ ]{E-mail addresses}
\makeatletter
\renewcommand\AB@affilsepx{, \protect\Affilfont}
\makeatother
\affil[1]{pierre-antoine.bernard@umontreal.ca}
\affil[2]{abderahmane.bouziane@umontreal.ca}
\affil[3]{samuel.pellerin@umontreal.ca}
\affil[4]{simone\textunderscore tetu@brown.edu}
\affil[5]{tsujimoto.satoshi.5s@kyoto-u.jp}
\affil[6]{luc.vinet@umontreal.ca}
\affil[7]{mzaimi@perimeterinstitute.ca}
\affil[8]{zhedanov@yahoo.com}
}

\date{\today}

\begin{document}

\maketitle
\begin{center}
This paper is cordially dedicated to Mourad Ismail on the occasion of his 80th birthday.
\end{center}
\vspace{3mm}
\begin{center}
\textbf{Abstract}\vspace{5mm}\\
\begin{minipage}{15cm}
A unified algebraic interpretation of both finite families of orthogonal polynomials and biorthogonal rational functions of $q$-Hahn type is provided. The approach relies on the meta $q$-Hahn algebra and its finite-dimensional bidiagonal representations. The functions of $q$-Hahn type are identified as overlaps (up to global factors) between bases solving ordinary or generalized eigenvalue problems in the representation of the meta $q$-Hahn algebra. Moreover, (bi)orthogonality relations, recurrence relations, difference equations and some contiguity relations satisfied by these functions are recovered algebraically using the actions of the generators of the meta $q$-Hahn algebra on various bases.
\end{minipage}
\end{center}

\medskip

\begin{center}
\begin{minipage}{15cm}

\textbf{MSC2020 database:} 	33D80; 33D45

\end{minipage}
\end{center}

\vspace{15mm}
\newpage

\section{Introduction}

This paper provides a new chapter of the broad program whose purpose is to extend the algebraic interpretation of the orthogonal polynomials (OPs) of the Askey scheme to biorthogonal rational functions (BRFs) in a unified manner, using meta algebras.  
Following \cite{TVZ24}, where terminating ${}_3F_2$ hypergeometric series are treated (the Hahn case), the present article focuses on their $q$-analogs, that is, on terminating ${}_3\phi_2$ basic hypergeometric series (the $q$-Hahn case).  

All families of OPs in the Askey scheme \cite{Koekoek} satisfy a recurrence relation and a difference equation  which can be viewed as eigenvalue problems (EVPs) for two operators, with the polynomials as solutions. The two bispectral operators satisfy the Askey--Wilson algebra, originally introduced in \cite{Zhe91}, or one of its limits or specializations. In this algebraic framework, the OPs are interpreted as overlap coefficients between eigenbases of the two generators of the Askey--Wilson algebra in some representation. In the situation of finite-dimensional representations, the two generators form a Leonard pair \cite{Ter01}. 
Such pairs are in correspondence with the terminating branch of the Askey scheme.  

The algebraic picture that is emerging for the BRFs extending the Askey scheme is as follows. The rational functions are solutions to two generalized eigenvalue problems (GEVPs) involving three operators, and the study of the algebraic relations satisfied by these operators has led to the introduction of meta algebras. The name is justified by the fact that these algebras contain those of Askey--Wilson type as well as the algebras associated with the generalized bispectral operators of the rational functions, and can hence provide an algebraic interpretation of both OPs and BRFs. The rational functions and polynomials appear as overlap coefficients when considering bases related to GEVPs or EVPs in representations of the meta algebra. Let us mention that the most general explicitly known ``elliptic" BRFs are related to elliptic quadratic Sklyanin algebras, as was noticed by Rains and Rosengren \cite{Rosen}.

The algebraic interpretation of BRFs started in \cite{TVZ21}, where the Hahn family was studied using a triplet of difference operators $X,Y,Z$, and the rational Hahn algebra was provided. The meta Hahn algebra was then introduced in \cite{VZ_Hahn} as a simpler algebra with three generators $X,V,Z$. The Hahn algebra is recovered from the pair formed by the generator $V$ and the linear pencil $W=X+\mu Z$, while the rational Hahn algebra is recovered from the triplet $X$, $Y=XV$ and $Z$. The bispectral properties and (bi)orthogonality of the overlap functions between GEVP and EVP bases were characterized in \cite{VZ_Hahn} with algebraic methods, but the rational functions of Hahn type and the Hahn polynomials were explicitly obtained only by using a differential or difference realization of the meta algebra. 

In the subsequent paper \cite{TVZ24}, a model-independent approach started being developed. It consists of first obtaining bidiagonal representations for the three generators $X,V,Z$ of the meta algebra. Then, the strategy is to compute the GEVP bases associated with the elements $X,Z$ and the EVP bases associated with the elements $W,V$, in terms of the bidiagonal representation bases. The OPs and BRFs are then identified explicitly as overlap coefficients between different (generalized) eigenbases, and their properties are determined using the actions of the generators in different bases. Section 2 of \cite{TVZ24} provides a more detailed summary of this general framework. This strategy was carried out in \cite{TVZ24} in the case of the OPs and BRFs of Hahn type.

The goal of this paper is to pursue this program for OPs and BRFs of $q$-Hahn type. Here, the starting point is the abstract meta $q$-Hahn algebra, which was first introduced at the end of \cite{BGVZ} as the algebra realized by a triplet of $q$-difference operators used to study the BRFs of $q$-Hahn type. 

The structure of the paper is as follows. In Section \ref{sec:metaqHahn}, the meta $q$-Hahn algebra is defined, and its connection with the $q$-Hahn algebra is explained. In Section \ref{sec:representations}, the finite-dimensional bidiagonal representations of the generators of the meta $q$-Hahn algebra are obtained. Several GEVPs and EVPs of interest are solved in Section \ref{sec:eigenbases} using the bidiagonal representations. In Section \ref{sec:actionbases}, the representations of the meta $q$-Hahn algebra on EVP bases are provided and are seen to be of tridiagonal form. The $q$-Hahn polynomials are explicitly identified in Section \ref{sec:qHahnOP} within the overlap coefficients between two EVP bases. Moreover, the orthogonality relation and bispectral properties of the $q$-Hahn polynomials are simply recovered from the properties of the eigenbases and the representations of the meta $q$-Hahn algebra on these bases. This provides a concise review of the characterization of the $q$-Hahn polynomials and their duals. In Section \ref{sec:qHahnBRF}, the rational functions of $q$-Hahn type are identified within the overlap coefficients between GEVP and EVP bases, and their biorthogonality relation as well as generalized bispectral properties are recovered with algebraic methods using the representations of the meta $q$-Hahn algebra. Both biorthogonal partner families are characterized. 
Section \ref{sec:conclusion} contains concluding remarks. A list of useful actions of elements of the meta $q$-Hahn algebra on several bases is provided in Appendix \ref{sec:appendix}.  

Mourad Ismail's oeuvre on the theory of orthogonal polynomials, special functions and their applications is monumental. He has in particular done pioneering work in the study of biorthogonal rational functions, a field that he keeps advancing as shown by his recent contribution \cite{Ismail} to the Indagationes Mathematicae volume dedicated to Tom Koornwinder. Mourad has also mentored and supported many young researchers throughout his career. The authors of this paper, who comprise a number of students and early career researchers, are hence very happy to praise Mourad Ismail by dedicating this paper to him.

\section{Meta $q$-Hahn algebra} \label{sec:metaqHahn}

The starting point for the algebraic interpretation of both polynomials and rational functions of $q$-Hahn type is the meta $q$-Hahn algebra, defined next.
Throughout this paper, we use the notations
\begin{equation}
[A,B]_q=AB-qBA, \qquad \{A,B\}=AB+BA, \qquad [x]_q=\frac{1-q^x}{1-q}.
\end{equation}
\begin{defi} \label{def:mqh}
The meta $q$-Hahn algebra $\mqh$ is the associative algebra with unit $I$ and generators $X,V,Z$ obeying the defining relations
\begin{align}
    &[Z,X]_q= Z^2+Z-(1-q)X, \label{eq:mqhrel1}\\
    &[X,V]_q = \{V,Z\}+V+\xi I, \label{eq:mqhrel2}\\
    &[V,Z]_q = (1+q)X-(1-q)V+\eta I, \label{eq:mqhrel3}
\end{align}
where $\xi,\eta$ are central parameters.
\end{defi}
It is apparent that in the limit $q\to 1$, one recovers the meta Hahn algebra as defined in \cite{TVZ24} from the meta $q$-Hahn algebra $\mqh$. As a consequence, the results obtained in the present paper are $q$-analogs of those obtained in \cite{TVZ24}, meaning that the latter can usually be recovered from the former in the limit $q \to 1$.

There is a Casimir element for the algebra $\mqh$ given by
\begin{align}
Q &= \left(q\eta -1 \right) X +(1-q) V +  (q\xi-\eta) Z + qX^2 -q(1-q)XV \nonumber \\
&\quad -(1+q)XZ + (2-q)VZ -q(1-q)XVZ + VZ^2.
\end{align}

Let us now discuss the important connection between the meta $q$-Hahn algebra and the $q$-Hahn algebra. The latter is realized by the recurrence and $q$-difference operators of the $q$-Hahn polynomials. With affine transformations of the generators, it can be defined in the following generic form (see also \cite{BVZ} for a slightly different presentation).
\begin{defi} \label{def:qh}
    The $q$-Hahn algebra $\qh$ is the unital associative algebra with generators $K_1,K_2$ and central elements $r_1,r_2,r_3,r_4,r_5,r_6$ obeying the defining relations
    \begin{align}
        &[K_2,[K_1,K_2]_q]_q = r_1(q-1)K_2^2 + r_2\{K_1,K_2\} + r_3 K_2 +r_4 K_1 +r_5 , \label{eq:qhrel1}\\
        &[[K_1,K_2]_q,K_1]_q = r_1(q-1)\{K_1,K_2\} + r_2K_1^2 - r_1^2 K_2 +r_3 K_1 +r_6 . \label{eq:qhrel2}
    \end{align}
\end{defi}
There is an embedding $\qh \xhookrightarrow{} \mqh$ of the $q$-Hahn algebra into the meta $q$-Hahn algebra given by the following mappings for the generators
\begin{align}
    K_1 \mapsto X+\mu Z, \qquad K_2 \mapsto V, \qquad \mu \in \mathbb{R}, 
\end{align}
and the following mappings for the central elements
\begin{align}
&r_1 \mapsto 1 + \mu(1-q), \quad 
r_2 \mapsto [2]_q, \quad 
r_3 \mapsto \xi(1-q)+2\eta+ \mu([2]_q+\eta(1-q)), \quad r_4 \mapsto 0, \\ 
&r_5 \mapsto [2]_q\mu\xi, \quad
r_6 \mapsto q^{-2}(1+\mu(1-q))\left(\eta(1+q^2\mu)-q\xi-[2]_qQ\right).
\end{align}
Note the presence of the Casimir element $Q$ of the meta $q$-Hahn algebra in the image of the central element $r_6$ of the $q$-Hahn algebra. The above embedding implies that in a finite-dimensional representation of the meta $q$-Hahn algebra, the linear pencil $X+\mu Z$ and the element $V$ form a Leonard pair of $q$-Hahn type \cite{Ter01}. This means that there is a basis where $V$ is diagonal and $X+\mu Z$ is irreducible tridiagonal, and there is a basis where $V$ is irreducible tridiagonal and $X+\mu Z$ is diagonal. The change of basis coefficients are expressed in terms of $q$-Hahn polynomials. This feature will prove essential in Section \ref{sec:qHahnOP}, where the $q$-Hahn polynomials are reviewed from the perspective of the meta $q$-Hahn algebra. In the rest of the paper, it will be convenient to define the linear pencil in the form $W=X-[\mu]_qZ$. In order to compare with the results of \cite{TVZ24} when $q\to1$, one must change $\mu \to -\mu$.

\begin{rem}
Originally, the meta $q$-Hahn algebra was introduced in \cite{BGVZ} as the algebra generated by elements $\Xt,\Vt,\Zt$ obeying the relations
\begin{align}
    &[\Xt,\Zt]_q= -\Zt^2-\Zt-(1-q)\Xt, \label{eq:mqhrel12}\\
    &[\Vt,\Xt]_q = -\{\Vt,\Zt\}+\chi_1\Xt-\Vt-q^{-1}\chi_1 \Zt-\chi_0 I, \label{eq:mqhrel22}\\
    &[\Zt,\Vt]_q = -\chi_2\Xt-(1-q)\Vt+\chi_1\Zt+\chi_3 I. \label{eq:mqhrel32}
\end{align}
This algebra is realized by the triplet of bispectral operators associated with the biorthogonal rational functions of $q$-Hahn type, with given values of the parameters $\chi_0,\chi_1,\chi_2,\chi_3$. In order to recover the relations \eqref{eq:mqhrel12}--\eqref{eq:mqhrel32} from Definition \ref{def:mqh}, one must first change the parameter $q\to q^{-1}$ in \eqref{eq:mqhrel1}--\eqref{eq:mqhrel3} and then take
\begin{equation}
\Xt = c_1 X +d_1 Z, \quad \Vt = c_2 V +d_2 I, \quad \Zt=Z, 
\end{equation}
with
\begin{align}
    &c_1=q^{-1}(1+(1-q)d_1), \quad d_1=\frac{\chi_1 - (1-q)d_2}{\chi_2}, \\
    &c_2=\frac{\chi_2}{[2]_q}c_1, \quad d_2= \frac{\chi_2 + (1-q)\chi_1 - \chi_2\sqrt{1 + q^{-1}(1-q^2)\chi_1/\chi_2}}{(1-q)^2}. 
\end{align}
The identity $[A,B]_q=(1-q)\{A,B\}-[B,A]_q$ is useful for obtaining relation \eqref{eq:mqhrel22}. Note that if 
\begin{equation}
\frac{\chi_1}{\chi_2}=-\frac{q[2\omega]_q}{[2]_q} \label{eq:relchi}
\end{equation}
for some $\omega$ (which happens in the realization of the meta $q$-Hahn algebra appearing in \cite{BGVZ}), then the above coefficients simplify to
\begin{align}
    &c_1=q^{\omega-1}, \quad d_1=-[\omega]_q, \quad c_2=\frac{\chi_2}{[2]_q}q^{\omega-1}, \quad d_2= \frac{\chi_2}{[2]_q} \left[\omega+1\right]_q\left[\omega\right]_q. 
\end{align} 
In this paper, we have chosen to work with the relations \eqref{eq:mqhrel1}--\eqref{eq:mqhrel3} instead of the original relations \eqref{eq:mqhrel12}--\eqref{eq:mqhrel32} since there are fewer parameters in the former. Moreover, the change $q\to q^{-1}$ is justified later by the identification of the $q$-Hahn polynomials as overlap coefficients between bases of interest. This change does not affect the embedding of the $q$-Hahn algebra into the meta $q$-Hahn algebra since the following identities hold
\begin{equation}
[[A,B]_q,A]_q=[A,[B,A]_q]_q = q^2 [A,[B,A]_{q^{-1}}]_{q^{-1}}. \label{eq:idqcomm} 
\end{equation}
Indeed, using \eqref{eq:idqcomm}, it is seen that the relations \eqref{eq:qhrel1}--\eqref{eq:qhrel2} with parameter $q^{-1}$ are the same as those with parameter $q$ up to a redefinition of the central elements $r_i$.
\end{rem}

\section{Bidiagonal finite-dimensional representations} \label{sec:representations}

The meta $q$-Hahn algebra admits finite-dimensional representations in which the generators act in a bidiagonal fashion on some basis. More precisely, we will take $X$ and $Z$ to be represented by lower bidiagonal matrices (raising action) and $V$ by an upper bidiagonal matrix (lowering action).

Let $N$ be a positive integer and $\V$ be a vector space of dimension $N+1$. The real scalar product of two vectors $\ket{v}$ and $\ket{w}$ in $\V$ will be denoted by $\braket{v|w}$. Let $\{ \ket{n} \}_{n=0}^N$ be a set of basis vectors of $\V$ such that $\braket{m|n}=\delta_{mn}$ for $m,n=0,1,\dots,N$; this will be referred to as the standard basis.

\begin{prop}
    There is a bidiagonal representation of the meta $q$-Hahn algebra $\mqh$ on the $N+1$-dimensional vector space $\V$ given by the following actions of the generators $Z,X,V$ on the basis vectors $\ket{n}$, for $n=0,1,\dots,N$:
    \begin{align}
    &Z \ket{n} = - \ket{n} + a_n \ket{n+1}, \label{eq:actZn} \\
    &X \ket{n} = -[\alpha-n]_q \ket{n} + a_n [\beta-n]_q \ket{n+1}, \label{eq:actXn} \\
    &V \ket{n} =  [n-\beta]_q  [\beta-n+1]_q  \ket{n} - \frac{q^{\alpha-N} \Q{n}\Q{N+1-n}}{a_{n-1}} \ket{n-1}, \label{eq:actVn}
\end{align}
where $a_n$ are constants related to the normalization of the basis vectors, with $a_{-1}=a_N=0$, and $\alpha,\beta$ are two parameters related as follows to the parameters $\xi$ and $\eta$ of the algebra $\mqh$:  
\begin{align}
    &\xi =  q^{\alpha+1}\Q{-\beta-1}\Q{\beta-N}, \label{eq:n0} \\
    &\eta =   q\Q{\alpha} + \Q{\alpha-N}. \label{eq:eta3}
\end{align}
\end{prop}
\begin{proof}
The bidiagonal raising actions of $Z$ and $X$ on $\V$ are of the form (for $n=0,1,\dots,N$)
\begin{align}
    &Z\ket{n} = c_n \ket{n} + a_n\ket{n+1},\\ 
    &X\ket{n} = d_n \ket{n} + b_n\ket{n+1},
\end{align}
for some constants $a_n,b_n,c_n,d_n$ to be determined, with $a_{-1}=a_N=0$ and $b_{-1}=b_N=0$ as truncation conditions. Using these actions and the defining relation \eqref{eq:mqhrel1} of $\mqh$, one finds the following constraint on the constants $a_n$ and $b_n$:
\begin{align}
q a_nb_{n+1}- a_{n+1}b_n+ a_na_{n+1}=0.
\end{align}
The solution to this recurrence equation for $b_n/a_n$ is given by
\begin{equation}
b_n=a_n\left([-n]_q+q^{-n}\frac{b_0}{a_0}\right).
\end{equation}
By reparametrizing $b_0/a_0=[\beta]_q$, the solution is written as $b_n=a_n[\beta-n]_q$, with $\beta$ an arbitrary constant and with $a_{-1}=a_N=0$. Relation \eqref{eq:mqhrel1} also leads to the following constraint on the constant $c_n$:
\begin{align}
(1 + c_n) (c_n - (1 - q) d_n) =0. \label{eq:constraintcn}
\end{align}
The previous equation is solved for $c_n=-1$. With this choice, relation \eqref{eq:mqhrel1} finally leads to the following constraint on the constant $d_n$:
\begin{equation}
 qd_{n+1} - d_n - 1 =0.
\end{equation}
Similarly, with the reparametrization $d_0=-[\alpha]_q$, where $\alpha$ is an arbitrary constant, the solution can be written as $d_n=-[\alpha-n]_q$. This proves equations \eqref{eq:actZn} and \eqref{eq:actXn}.

The bidiagonal lowering action of $V$ has the form
\begin{equation}
V \ket{n} = e_n \ket{n} + f_n \ket{n-1},
\end{equation}
where the constant $f_n$ must obey the truncation conditions $f_0=f_{N+1}=0$. Using relation \eqref{eq:mqhrel3} together with the bidiagonal actions of $X,V,Z$, one finds the following constraint on the constants $f_n$ and $a_n$:
\begin{align}
a_nf_{n+1} -qa_{n-1}f_n + (1+q)[\alpha-n]_q-\eta=0.
\end{align}
With the condition $a_{-1}f_0=0$, the previous recurrence relation is solved by
\begin{equation}
a_{n-1}f_n = [n]_q(\eta-[\alpha-n+1]_q-q[\alpha]_q).
\end{equation}
One must impose relation \eqref{eq:eta3} in order for the solution to be consistent with the condition $a_Nf_{N+1}=0$. Therefore, using \eqref{eq:eta3}, the solution can be written as
\begin{equation}
f_n = -\frac{q^{\alpha-N}[n]_q[N+1-n]_q}{a_{n-1}}.
\end{equation}
Note that with this presentation of the action of $V$, the factor $[n]_q/a_{n-1}$ must vanish in the limit $n\to 0$ so that $f_0=0$, and the factor $[N+1-n]_q/a_n$ must vanish in the limit $n\to N$ so that $f_{N+1}=0$.

Relations \eqref{eq:mqhrel2} and \eqref{eq:mqhrel3} lead to two constraints for the constant $e_n$:
\begin{align}
&e_{n+1}-qe_n - (1+q)[\beta-n]_q=0,
&[\beta-n+1]_qe_{n+1}-q[\beta-n-1]_qe_n=0.
\end{align}
The solution is given by $e_n=[n-\beta]_q[\beta-n+1]_q$, which completes the proof of \eqref{eq:actVn}. There is a final constraint imposed by the defining relation \eqref{eq:mqhrel2} of $\mqh$ which reduces to \eqref{eq:n0}.
\end{proof}

\begin{rem}
The solution $c_n=(1-q)d_n$ to equation \eqref{eq:constraintcn} will not be considered in this paper, as it leads (see Proposition \ref{prop:(G)EVPsol}) to a GEVP between $X$ and $Z$ with degenerate eigenvalue $\lambda=(1-q)^{-1}$. In the limit $q\to 1$, this solution becomes $c_n=0$, and the GEVP is trivial unless $\alpha \in \{0,1,\dots,N\}$, in which case the eigenvalue $\lambda$ can take any value.
\end{rem}

The transpose $O^\top$ of an operator $O$ acting on $\V$ is defined as usual by $\braket{v|O^\top|w}= \braket{w|O|v}$, where $\ket{v},\ket{w}$ are any vectors in $\V$. We record for convenience the action of the transposed operators $Z^\top,X^\top,V^\top$ on the basis vectors $\ket{n}$ for $n=0,1,\dots,N$: 
\begin{align}
    &Z^\top \ket{n} = - \ket{n} + a_{n-1} \ket{n-1}, \label{eq:actZTn} \\
    &X^\top \ket{n} = -[\alpha-n]_q \ket{n} + a_{n-1} [\beta-n+1]_q \ket{n-1}, \label{eq:actXTn} \\
    &V^\top \ket{n} =  [n-\beta]_q [\beta-n+1]_q  \ket{n} - \frac{q^{\alpha-N} \Q{n+1}\Q{N-n}}{a_{n}} \ket{n+1}. \label{eq:actVTn}
\end{align}
Note that we use the same notation for the abstract elements of the algebra $\mqh$ and their representation as operators acting on the vector space $\V$. The transpose only makes sense for the operators.

\section{(Generalized) Eigenbases} \label{sec:eigenbases}

Following the general approach described in the introduction (see also \cite{TVZ24}), the next important step for obtaining an algebraic interpretation of the special functions of $q$-Hahn type is to consider the bases which solve certain (generalized) eigenvalue problems within the representation of $\mqh$ on the vector space $\V$. The use of the bidiagonal actions obtained in the previous section is very helpful for obtaining explicitly these bases and the associated (generalized) eigenvalues; this is the advantage of the methods developed previously in \cite{TVZ24} in the Hahn case and extended here to the $q$-Hahn case.   

\begin{prop}\label{prop:(G)EVPsol}
The solutions to the GEVPs and EVPs of interest are given as follows in terms of expansions over the basis $\{\, \ket{\ell}\}_{\ell = 0}^N$ of the representation space $\V$, for $n=0,1,\dots,N$:
\begin{itemize}[leftmargin=*]
\item $X \ket{d_n} = \lambda_n Z \ket{d_n}$
\begin{align}
    \lambda_n &=[\alpha-n]_q, \label{eq:lambda}\\
    \ket{d_n} &= \sum_{\ell=0}^{N} \dfrac{a_n a_{n+1}\dots a_{N-1}}{a_{\ell} a_{\ell+1}\dots a_{N-1}} \dfrac{(q^{n-N};q)_{N-\ell}}{(q^{n-N};q)_{N-n}} \dfrac{(q^{n-N-\alpha+\beta+1}; q)_{N-n}}{(q^{n-N-\alpha+\beta+1}; q)_{N-\ell}} \ket{\ell};
    \label{eq:dn}
\end{align}

\item $X^{\top} \ket{d_n^*} = \lambda_n Z^{\top} \ket{d_n^*}$
\begin{align}
    \lambda_n &=[\alpha-n]_q,\\
    \ket{d_n^*} &= \sum_{\ell=0}^{N} \dfrac{a_0 a_{1}\dots a_{n-1}}{a_{0} a_{1}\dots a_{\ell-1}} \frac{q^{(\beta-\alpha)n}}{q^{(\beta-\alpha)\ell}} \dfrac{(q^{-n};  q)_{\ell}}{(q^{-n};  q)_{n}} \dfrac{(q^{-n+\alpha-\beta};  q)_{n}}{(q^{-n+\alpha-\beta};  q)_{\ell}} \ket{\ell};
    \label{eq:dn*}
\end{align}

\item $V \ket{e_n} = \nu_n \ket{e_n}$
\begin{align}
    \nu_n &= [n-\beta]_q [\beta-n+1]_q,
    \label{mu} \\
    \ket{e_n} &= \sum_{\ell=0}^{N} \dfrac{a_0 a_{1}\dots a_{\ell-1}}{a_{0} a_{1}\dots a_{n-1}} \frac{q^{(\alpha-\beta-1)n}}{q^{(\alpha-\beta-1)\ell}} \dfrac{(q^{-n};  q)_{\ell}}{(q^{-n};  q)_{n}} \dfrac{(q;  q)_{n}}{(q;  q)_{\ell}} \dfrac{(q^{-N};  q)_{n}}{(q^{-N};  q)_{\ell}} \dfrac{(q^{n-2\beta-1};  q)_{\ell}}{(q^{n-2\beta-1};q)_n} \ket{\ell}; 
    \label{eq:en}
\end{align}
    
\item $V^{\top} \ket{e_n^*} = \nu_n e_n^*$
\begin{align}
    \nu_n &= [n-\beta]_q [\beta-n+1]_q, \\
    \ket{e_n^*} &= \sum_{\ell=0}^{N} \dfrac{a_\ell a_{\ell+1}\dots a_{N-1}}{a_{n} a_{n+1}\dots a_{N-1}} \frac{q^{(\alpha+\beta-N)(N-n)}}{q^{(\alpha+\beta-N)(N-\ell)}} \dfrac{(q^{n-N};  q)_{N-\ell}}{(q^{n-N};  q)_{N-n}} \dfrac{(q;  q)_{N-n}}{(q;  q)_{N-\ell}} \dfrac{(q^{-N};  q)_{N-n}}{(q^{-N};  q)_{N-\ell}} \dfrac{(q^{-N-n+2\beta+1};  q)_{N-\ell}}{(q^{-N-n+2\beta+1};  q)_{N-n}} \ket{\ell}; 
    \label{eq:en*}
\end{align}

\item
$ W\ket{f_n} = (X-\Q{\mu} Z) \ket{f_n} = \rho_n \ket{f_n}$ 
\begin{align}
    \rho_n &= -\Q{\alpha-n}+\Q{\mu}, \label{eq:rhon}\\
    \ket{f_n} &= \sum_{\ell=0}^{N} \dfrac{a_n a_{n+1}\dots a_{N-1}}{a_{\ell} a_{\ell+1}\dots a_{N-1}} \frac{q^{(\alpha-\mu-n)(N-\ell)}}{q^{(\alpha-\mu-n)(N-n)}} \dfrac{(q^{n-N};  q)_{N-\ell}}{(q^{n-N};  q)_{N-n}} \dfrac{(q^{\beta-\mu-N+1};  q)_{N-n}}{(q^{\beta-\mu-N+1};  q)_{N-\ell}} \ket{\ell};
    \label{eq:fn}
\end{align}

\item 
$W^{\top}\ket{f_n^*} = (X^{\top}-\Q{\mu} Z^{\top}) \ket{f_n^*} = \rho_n \ket{f_n^*}$
\begin{align}
    \rho_n &= -\Q{\alpha-n}+\Q{\mu}, \label{eq:rho}\\
    \ket{f_n^*} &= \sum_{\ell=0}^{N} \dfrac{a_0 a_{1}\dots a_{n-1}}{a_{0} a_{1}\dots a_{\ell-1}} \frac{q^{(\beta-\alpha)n}}{q^{(\beta-\alpha)\ell}} \dfrac{(q^{-n};  q)_{\ell}}{(q^{-n};  q)_{n}} \dfrac{(q^{\mu-\beta};  q)_{n}}{(q^{\mu-\beta};  q)_{\ell}} \ket{\ell};
    \label{eq:fn*}
\end{align}
\end{itemize} 
where the $q$-Pochhammer symbol is defined as usual by
\begin{equation}
(a;q)_0:= 1, \qquad (a;q)_k:=\prod_{i = 0}^{k-1}(1-aq^i), \quad k=1,2,3,\dots   
\end{equation}
\end{prop}
\begin{proof}
The generalized eigenvalues $\lambda_n$ and ordinary eigenvalues $\nu _n$ and $\rho _n$ are all easily determined from the diagonal terms of the bidiagonal actions \eqref{eq:actZn}--\eqref{eq:actVn} of the generators $Z,X,V$ of $\mqh$.

Consider the EVP for $V$ in the $N+1$-dimensional representation space $\V$, with eigenvalues $\nu_{n} =[n-\beta]_q[\beta-n+1]_q $ and associated eigenvectors $\ket{e_n}$, for $n=0,1,\dots,N$. For a fixed $n$, the eigenvector $\ket{e_n}$ can be expanded on the orthonormal basis $\{\ket{\ell}\}_{\ell=0}^N$ as 
\begin{equation}
    \ket{e_n}=\sum_{\ell=0}^N \braket{\ell \mid e_n} \ket{\ell}.
\end{equation}
From the eigenvalue equation $V\ket{e_n} = \nu_n \ket{e_n}$, one finds
\begin{equation}
  \braket{\ell \mid(V-\nu_n)\mid e_n}=0, \quad  (\ell = 0, 1, \ldots, N).
\end{equation}
Using the action \eqref{eq:actVTn} of $V^{\top}$ on $\bra{\ell}$ in the previous equation, one gets 
\begin{align}
  (\nu_\ell-\nu_n)\braket{\ell \mid e_n} - \frac{q^{\alpha-N} \Q{\ell+1}\Q{N-\ell}}{a_{\ell}}\braket{\ell+1\mid e_n}=0, \quad (\ell = 0, 1, \ldots, N).
  \label{eq:ebasecond}
\end{align}
Equation \eqref{eq:ebasecond} is a two-term recurrence relation for the coefficients $\braket{\ell \mid e_n}$ that can easily be solved up to a normalization constant. It will be convenient to choose the normalization condition $\braket{n\mid e_n}=1$. One thus arrives at the explicit expression \eqref{eq:en} for $\ket{e_n}$.

Consider now the GEVP for $X$ and $Z$, with generalized eigenvalues $\lambda_n = \Q{\alpha-n}$ and associated generalized eigenvectors $\ket{d_n}$, for $n=0,1,\dots,N$. From the generalized eigenvalue equation $X\ket{d_n} = \lambda_n Z \ket{d_n}$, one finds
\begin{align}
  \braket{N-\ell\mid(X-\lambda_n Z)\mid d_n}=0, \quad (\ell = 0, 1, \ldots, N),
\end{align}
which, with the help of \eqref{eq:actZTn} and \eqref{eq:actXTn}, leads to
\begin{equation}
  (\lambda_{N-\ell}-\lambda_{n})\braket{N-\ell \mid d_n} + (\lambda_n-\Q{\beta-N+\ell+1})a_{N-\ell-1}\braket{N-\ell-1 \mid d_n}=0, \quad (\ell = 0, 1, \ldots, N). 
  \label{eq:dbasecond}
\end{equation}
Again, this is a two-term recurrence relation for the coefficients $\braket{N-\ell \mid d_n}$. The solution to equation \eqref{eq:dbasecond} under the normalization condition $\braket{n \mid d_n}=1$ leads to the expression \eqref{eq:dn} for $\ket{d_n}$.    

For the remaining (generalized) eigenvectors, let us simply indicate that the expressions provided in \eqref{eq:en*}, \eqref{eq:dn*}, \eqref{eq:fn} and \eqref{eq:fn*} are obtained similarly from the following equations (for $\ell =0,1,\ldots,N$): 
\begin{align}
&\braket{N-\ell \mid (V^{\top}-\nu_n) \mid e_n^*} =0, & &\hspace{-2cm}\braket{\ell \mid(X^{\top}-\lambda_nZ^{\top}) \mid d_n^*} =0, \\
&\braket{N-\ell \mid (W-\rho_n) \mid f_n} =0, & &\hspace{-2cm}\braket{\ell \mid(W^{\top}-\rho_n) \mid f_n^*} =0,
\end{align}
and the normalization conditions $\braket{n \mid e_n^*}=\braket{n \mid d_n^*}=\braket{n \mid f_n}=\braket{n \mid f_n^*}=1$.
\end{proof}

Let us point out that the sums in the expressions \eqref{eq:dn*}, \eqref{eq:en} and \eqref{eq:fn*} for the vectors $\ket{d_n^*}$, $\ket{e_n}$ and $\ket{f_n^*}$ can, in fact, end at $\ell=n$ as the subsequent terms vanish, while the sums in the expressions \eqref{eq:dn}, \eqref{eq:en*} and \eqref{eq:fn} for the vectors $\ket{d_n}$, $\ket{e_n^*}$ and $\ket{f_n^*}$ can start at $\ell=n$ as the preceding terms vanish. This is similar to the Hahn case, see Section 5 of \cite{TVZ24}.

From the expression \eqref{eq:dn} for $\ket{d_n}$ and the action \eqref{eq:actZn} of $Z$ on the standard basis, it is straightforward to compute
\begin{equation}
       Z \ket{d_n} = 
       - \sum_{\ell=0}^{N} \dfrac{a_n a_{n+1}\dots a_{N-1}}{a_{\ell} a_{\ell+1}\dots a_{N-1}}\frac{(q^{n-N};q)_{N-\ell}(q^{n-N-\alpha+\beta+2};q)_{N-n}}{(q^{n-N};q)_{N-n}(q^{n-N-\alpha+\beta+2};q)_{N-\ell}} q^{n-\ell}\ket{\ell}. \label{Zdn}
\end{equation}
Again, the previous sum can start at $\ell=n$.

The following orthogonality relations hold for $m,n=0,1,\dots,N$:
\begin{gather}
\braket{e_m^\star \mid e_n} = \braket{e_n \mid e_m^\star}  = \delta_{mn}, \quad \braket{f_m^\star \mid f_n} = \braket{f_n \mid f_m^\star} = \delta_{mn}, \label{eq:orthef} \\
\braket{d_m^\star \mid Z \mid d_n} = \braket{d_n \mid Z^\top \mid d_m^\star} = -\delta_{mn}. \label{eq:orthd}
\end{gather}
The cases $m\neq n$ generically follow from the definition of the vectors as solutions to (generalized) eigenvalue problems with distinct (generalized) eigenvalues (see also Section 2 in \cite{TVZ24}), and the cases $m=n$ are a consequence of the normalization choices made in Proposition \ref{prop:(G)EVPsol}. From these previous orthogonality relations, one deduces that the following operators are all equal to the identity operator on $\V$:
\begin{gather}
\sum_{n=0}^N \ket{e_n}\bra{e_n^*}, \quad \sum_{n=0}^N \ket{e_n^*}\bra{e_n}, \quad \sum_{n=0}^N \ket{f_n}\bra{f_n^*}, \quad \sum_{n=0}^N \ket{f_n^*}\bra{f_n}, \label{eq:idef} \\
-\sum_{n=0}^N Z\ket{d_n}\bra{d_n^*}, \quad -\sum_{n=0}^N \ket{d_n^*}\bra{d_n}Z^\top. \label{eq:idd}
\end{gather}

\section{Representations of $\mqh$ in various bases} \label{sec:actionbases}

In this section, the actions of the generators of $\mqh$ on some of the bases obtained in the previous section are provided (see also Appendix \ref{sec:appendix} for additional formulas). These will prove useful for characterizing the recurrence and difference properties of the special functions of $q$-Hahn type. For an operator $O$ acting on the vector space $\V$, the notation $O_{m,n}^{(b)}$ will refer to its $m,n$ matrix entry in the basis $b \in \{d, d^*, e, e^*, f, f^*\}$, see Proposition \ref{prop:(G)EVPsol}. Note moreover that for $b \in \{e,f\}$, the resolutions of the identity \eqref{eq:idef} imply that
\begin{equation}
O^{\top (b^*)}_{m,n} = O_{n,m}^{(b)}. \label{eq:transposedmatrixentries}
\end{equation}

\subsection{Representation in the $e$ and $e^{*}$ bases.}\label{ssec:repe} 

The $e$ basis is formed of the eigenvectors $\ket{e_n}$ solving the eigenvalue equation $V\ket{e_n}=\nu_n\ket{e_n}$.
Using the explicit expression \eqref{eq:en} for the eigenvectors $\ket{e_n}$ as well as the actions \eqref{eq:actZn} and \eqref{eq:actXn} of $Z$ and $X$ on the standard basis, one finds  
\begin{subequations} \label{eq:actZe}
\begin{align}
    Z \ket{e_n} &= Z^{(e)}_{n+1,n} \ket{e_{n+1}} +Z^{(e)}_{n,n} \ket{e_{n}} + Z^{(e)}_{n-1,n} \ket{e_{n-1}},
\label{action:Zone}
\end{align}
where
\begin{align}
&Z^{(e)}_{n+1,n}=a_n, \label{action:Zone:coe1}\\
&Z^{(e)}_{n,n}=-1-\frac{q^{\alpha+\beta-2n}(q [n]_q [2\beta-n-N]_q+[n-N]_q [2\beta+2-n]_q)}{[2\beta-2n]_q [2\beta-2n+2]_q}, \label{action:Zone:coe2}  \\
&Z^{(e)}_{n-1,n}=q^{2\alpha-2\beta-1}\frac{\Q{n}\Q{n-N-1}\Q{2\beta-n+2}\Q{2\beta-n-N+1}}{a_{n-1}(\Q{2n-2\beta-2})^2\Q{2\beta-2n+1}\Q{2\beta-2n+3}},\label{action:Zone:coe3}
\end{align}
\end{subequations}
and
\begin{subequations} \label{eq:actXe}
\begin{align}
    X \ket{e_n} &=X^{(e)}_{n+1,n} \ket{e_{n+1}} + X^{(e)}_{n,n} \ket{e_{n}} + X^{(e)}_{n-1,n} \ket{e_{n-1}}, \label{action:Xone}
\end{align}
where
\begin{align}
&X^{(e)}_{n+1,n}=\Q{\beta-n}a_n,
\label{action:Xone:coe1}
\\ &X^{(e)}_{n,n}=\frac{-1}{1-q}\left(1-q^{\alpha-n}\frac{(1+q^{\beta-N})(1+q^{\beta+1})}{(1+q^{\beta-n})(1+q^{\beta-n+1})}\right),
\label{action:Xone:coe2}
\\
    &X^{(e)}_{n-1,n}= \Q{n-\beta-1}Z^{(e)}_{n-1,n}.
\label{action:Xone:coe3}
\end{align}
\end{subequations}
Therefore, the operator $V$ is diagonal in the basis $e$, while the operators $Z,X$ (and, consequently, the operator $W=X-[\mu]_qZ$) are tridiagonal. One can use the above results and relation \eqref{eq:transposedmatrixentries} to directly deduce the tridiagonal actions of the transposed operators $Z^\top$ and $X^\top$ on the $e^*$ basis, in which $V^\top$ is diagonal by definition. The representation of $\mqh$ on the $e$ and $e^*$ bases will play a role in the characterization of both OPs and BRFs of $q$-Hahn type.

\subsection{Representation in the $f$ and $f^{*}$ bases.} \label{ssec:repf} 
The $f$ basis is formed of the eigenvectors $\ket{f_n}$ solving the eigenvalue equation $W\ket{f_n}=\rho_n\ket{f_n}$ for the linear pencil $W=X-[\mu]_qZ$. Here, using the expression \eqref{eq:fn} for the eigenvectors $\ket{f_n}$ and the action \eqref{eq:actVn} of $V$ on the standard basis, one computes
\begin{subequations}\label{eq:actVf}
\begin{align}
    V \ket{f_n} &= V^{(f)}_{n+1,n} \ket{f_{n+1}} + V^{(f)}_{n,n} \ket{f_{n}} + V^{(f)}_{n-1,n} \ket{f_{n-1}}, \label{action:Vonf} 
\end{align}
where
\begin{align}
  &V^{(f)}_{n+1,n}=a_n q^{n-\alpha+1}\Q{n - \beta + \mu}\Q{n- N +\beta + \mu + 1}, \label{vf+1}\\
&V^{(f)}_{n,n}=
[n]_q[n+\mu-N]_q 
+q[n+\beta-N]_q[n+\mu-\beta]_q 
+q^{n}[-\beta]_q[\mu+1]_q, \label{vf}\\
&V^{(f)}_{n-1,n}=q^{\alpha-n+1}\dfrac{\Q{n}\Q{n - N - 1}}{a_{n-1}}.  \label{vf-1}
\end{align}
\end{subequations}
Hence, $W$ is diagonal in the basis $f$, while $V$ is tridiagonal. Again, one can use the formulas above and relation \eqref{eq:transposedmatrixentries} to obtain the tridiagonal action of the transposed operator $V^\top$ on the basis $f^*$ which diagonalizes $W^\top$.

The results of Subsections \ref{ssec:repe} and \ref{ssec:repf} confirm that the operators $V$ and $W$ form a Leonard pair, as expected from the embedding of the $q$-Hahn algebra in the meta $q$-Hahn algebra discussed in Section~\ref{sec:metaqHahn}.

\section{$q$-Hahn and dual $q$-Hahn polynomials} \label{sec:qHahnOP}

It is known that the $q$-Hahn and dual $q$-Hahn polynomials can be interpreted as overlap functions within representations of the $q$-Hahn algebra. Their orthogonality and bispectral properties can be found in \cite{Koekoek}. Here, in the same spirit as in the previous work \cite{TVZ24} treating the Hahn case, we explain how the (dual) $q$-Hahn polynomials can be recovered and characterized quite easily using the bidiagonal representation of the meta $q$-Hahn algebra, in an approach that unifies the study of both the OPs and the BRFs.     

With the usual notations for the basic hypergeometric functions ${}_r\phi_s$, the $q$-Hahn and dual $q$-Hahn polynomials are respectively defined as \cite{Koekoek} 
\begin{align}
    Q_m (q^{-x}; \hat{\alpha}, \hat{\beta}, N;q) = \qhg{3}{2}\qargu{q^{-m},\hat{\alpha} \hat{\beta} q^{m+1},q^{-x}}{\hat{\alpha} q,q^{-N}}{q}{q}, \quad m=0, \dots , N, \label{HP} \\
    R_m (\mu(x); \hat{\alpha}, \hat{\beta}, N;q) = \qhg{3}{2}\qargu{q^{-m},q^{-x}, \hat{\alpha} \hat{\beta} q^{x+1}}{\hat{\alpha} q,q^{-N}}{q}{q}, \quad m=0, \dots , N, \label{dHP}
\end{align}
where $\mu(x)=q^{-x}+\hat{\alpha} \hat{\beta} q^{x+1}$. These polynomials are related by 
\begin{equation}
R_m (\mu(n); \hat{\alpha}, \hat{\beta}, N;q)=Q_n (q^{-m}; \hat{\alpha}, \hat{\beta}, N;q), \quad m,n=0,\dots,N. \label{eq:HahndualHahn}
\end{equation}
The following notation for the $q$-Pochhammer symbol will be used:
\begin{equation}
(a_1,a_2,\dots,a_k;q)_n :=(a_1;q)_n(a_2;q)_n\dots(a_k;q)_n. 
\end{equation}

\subsection{Identification}
We start by identifying precisely the $q$-Hahn polynomials within the overlap coefficients between EVP bases in the representation of the meta $q$-Hahn algebra $\mqh$.
\begin{prop}
\label{propid}
The functions $S_m(n) := \braket{e_m \mid f_n^*}$ and $\Tilde{S}_m(n) := \braket{e_m^* \mid f_n}$ are both expressible in terms of $q$-Hahn polynomials as follows, for $m,n=0,1,\dots,N$:
\begin{align}
 S_m(n)&=
    \dfrac{a_{0}a_{1} \dots a_{n-1}}{a_0 a_1 \dots a_{m-1}}q^{m(\alpha-\beta-N+m-1)+n(\beta-\alpha+\frac{n}{2}+\frac{1}{2})}
   \dfrac{(q;q)_N (-1)^n}{(q;q)_n (q;q)_{N-m}}\dfrac{(\hat{\alpha} q;q)_{n}}{(\hat{\alpha} \hat{\beta} q^{m+1};q)_{m}} Q_m (q^{-n}; \hat{\alpha}, \hat{\beta}, N;q),\label{Shahn}\\
\tilde S_m(n)&=
    \frac{a_{n}a_{n+1} \dots a_{N-1}}{a_{m}a_{m+1}\dots a_{N-1}} \hat{\alpha}^{N-m-n} q^{m(\beta-\alpha+N-m)+n(\alpha-\beta-\frac{n}{2}-\frac{3}{2})+N} \nonumber \\
    & \quad \times \dfrac{(q;q)_N (-1)^n (\hat{\alpha} q;q)_m }{(q;q)_m (q;q)_{N-n} (\hat{\beta}q;q)_m} \dfrac{(\hat{\beta}q;q)_{N-n}}{(\hat{\alpha} \hat{\beta} q^{2m+2};q)_{N-m}} Q_m (q^{-n}; \hat{\alpha}, \hat{\beta}, N;q), \label{sthahn}
\end{align}
with
\begin{equation}
    \hat{\alpha} = q^{\mu-\beta-1},  \qquad \hat{\beta} = q^{-\mu-\beta-1}.\label{newpar}
\end{equation}
\end{prop}

\begin{proof}
Using the expressions \eqref{eq:en} and \eqref{eq:fn*} for the eigenvectors $\ket{e_n}$ and $\ket{f_n^*}$ together with the orthonormality relation $\braket{\ell \mid k} = \delta _{\ell k}$ of the standard basis, one finds 
\begin{align}
     S_m(n)=\braket{e_m \mid f_n^*} &=   \dfrac{a_{0}a_{1} \dots a_{n-1}}{a_0 a_1 \dots a_{m-1}}
     q^{m(\alpha-\beta-1)+n(\beta-\alpha)}\frac{(q^{\mu-\beta};q)_n (q, q^{-N};q)_m}{(q^{-n};q)_n (q^{-m},q^{m-2\beta-1};q)_m} \nonumber \\
    &\quad \times \sum_{\ell=0}^{N} \frac{(q^{-n},q^{-m},q^{m-2\beta-1};q)_{\ell} }{(q^{-N},q^{\mu-\beta};q)_{\ell}}\frac{q^\ell}{(q;q)_\ell}. \label{eq:emfsn}
\end{align}
Comparing the sum in \eqref{eq:emfsn} to the $q$-Hahn polynomials defined in \eqref{HP}, one straightforwardly arrives at the formula \eqref{Shahn} for the function $S_m(n)$ with the identification \eqref{newpar} for the parameters $\hat{\alpha}, \hat{\beta}$.

By using similarly the expressions \eqref{eq:en*} and \eqref{eq:fn} for the eigenvectors $\ket{e^*_n}$ and $\ket{f_n}$, and after a change of the summation index, one finds
 \begin{align}
    \tilde S_m(n)=\braket{e_m^* \mid f_n} &=
    \frac{a_{n}a_{n+1} \dots a_{N-1}}{a_{m}a_{m+1}\dots a_{N-1}} q^{(\alpha+\beta-N)(N-m)}q^{(n+\mu-\alpha)(N-n)} \nonumber\\
    &\quad \times \dfrac{(q^{-N+\beta-\mu+1};q)_{N-n} (q^{-N},q;q)_{N-m}}{(q^{n-N};q)_{N-n} (q^{m-N},q^{-N-m+2\beta+1};q)_{N-m}} \nonumber\\
   &\quad \times 
    \sum_{k=0}^{N}  \frac{(q^{m-N},q^{n-N},q^{-N-m+2\beta+1};q)_{k}}{(q^{-N},q^{-N+\beta-\mu+1};q)_k} \frac{q^{(N-\beta-\mu-n)k}}{(q;q)_k}.   \label{ST1}
 \end{align}
Here, the identification \eqref{newpar} can be used to rewrite the result \eqref{ST1} as
 \begin{align}
    \tilde S_m(n) &= \frac{a_{n}a_{n+1} \dots a_{N-1}}{a_{m}a_{m+1}\dots a_{N-1}} 
     \hat{\alpha}^{m-n}\hat{\beta}^{m-N}q^{(\alpha-\beta-N-2)(N-m)+(\beta-\alpha+n+1)(N-n)} \nonumber\\
    & \quad \times \dfrac{(\hat{\alpha}^{-1}q^{-N};q)_{N-n} (q^{-N},q;q)_{N-m}}{(q^{n-N};q)_{N-n} (q^{m-N},\hat{\alpha}^{-1}\hat{\beta}^{-1}q^{-N-m-1};q)_{N-m}} \nonumber\\
   &\quad \times 
   \qhg{3}{2} \qargu{q^{m-N}, q^{n-N}, \hat{\alpha}^{-1}\hat{\beta}^{-1}q^{-N-m-1}}{q^{-N}, \hat{\alpha}^{-1}q^{-N}}{q}{\hat{\beta}q^{N-n+1}}. \label{ST1_2}
 \end{align}
With the help of Sears' formula (equation (3.2.5) in \cite{GR}) 
\begin{equation}
     \qhg{3}{2}\qargu{q^{-n}, a, b}{d, e}{q}{\frac{deq^n}{ab}} = \frac{(e/a;q)_n}{(e;q)_n} \qhg{3}{2}\qargu{q^{-n}, a, d/b}{d, aq^{1-n}/e}{q}{q}, \quad n=0,1,2,\dots, \label{TF1}
\end{equation}
one can transform the $q$-hypergeometric sum in \eqref{ST1_2} as 
\begin{align}
    &\qhg{3}{2} \qargu{q^{m-N}, q^{n-N}, \hat{\alpha}^{-1}\hat{\beta}^{-1}q^{-N-m-1}}{q^{-N}, \hat{\alpha}^{-1}q^{-N}}{q}{\hat{\beta}q^{N-n+1}} \nonumber \\
    &\quad =\frac{(\hat{\alpha}^{-1}q^{-m};q)_{N-n}}{(\hat{\alpha}^{-1}q^{-N};q)_{N-n}} \qhg{3}{2} \qargu{q^{m-N}, q^{n-N}, \hat{\alpha}\hat{\beta}q^{m+1}}{q^{-N}, \hat{\alpha}q^{m+n-N+1}}{q}{q}. \label{eq:transf1}
\end{align}
Then, with a double application of the transformation formula (equation (3.2.3) in \cite{GR})
\begin{equation}
      \qhg{3}{2}\qargu{q^{-n}, b, c}{d, e}{q}{q} = \frac{(de/bc;q)_n}{(e;q)_n} \left(\frac{bc}{d}\right)^n \qhg{3}{2}\qargu{q^{-n}, d/b, d/c}{d, de/bc}{q}{q}, \quad n=0,1,2,\dots,  \label{TF2}
\end{equation}
one can rewrite the $q$-hypergeometric function in the RHS of \eqref{eq:transf1} as
\begin{align}
    \qhg{3}{2} \qargu{q^{m-N}, q^{n-N}, \hat{\alpha}\hat{\beta}q^{m+1}}{q^{-N}, \hat{\alpha}q^{m+n-N+1}}{q}{q} 
    &=(\hat{\alpha}\hat{\beta}q^{m+1})^{N-m-n} \dfrac{(\hat{\beta}^{-1}q^{n-m-N};q)_{N-n} (\hat{\alpha}q;q)_m}{(\hat{\alpha}q^{m+n-N+1};q)_{N-n} (\hat{\beta}^{-1}q^{n-m-N};q)_{m}} \nonumber\\
    &\quad \times \qhg{3}{2} \qargu{q^{-m}, q^{-n}, \hat{\alpha}\hat{\beta}q^{m+1}}{q^{-N}, \hat{\alpha}q}{q}{q}. \label{eq:transf2}
\end{align}
The $q$-hypergeometric function in the RHS of \eqref{eq:transf2} can now easily be written in terms of the $q$-Hahn polynomials defined in \eqref{HP}.
Combining the results \eqref{ST1_2},  \eqref{eq:transf1} and \eqref{eq:transf2}, one obtains
\begin{align}
    \tilde S_m(n) &= \frac{a_{n}a_{n+1} \dots a_{N-1}}{a_{m}a_{m+1}\dots a_{N-1}} \hat{\alpha}^{N-2n}\hat{\beta}^{-n} 
     q^{(\alpha-\beta-N-2)(N-m)+(\beta-\alpha+n+1)(N-n)+(m+1)(N-m-n)} \nonumber\\
    & \times  \dfrac{(q^{-N},q;q)_{N-m} (\hat{\alpha}^{-1}q^{-m},\hat{\beta}^{-1}q^{-N+n-m};q)_{N-n} (\hat{\alpha}q;q)_m}{(q^{m-N},\hat{\alpha}^{-1}\hat{\beta}^{-1}q^{-N-m-1};q)_{N-m}(q^{n-N},\hat{\alpha}q^{m+n-N+1};q)_{N-n} (\hat{\beta}^{-1}q^{n-m-N};q)_{m}} Q_m (q^{-n}; \hat{\alpha}, \hat{\beta}, N;q). \label{ST1_3}
 \end{align}
 In order to arrive at the formula \eqref{sthahn}, it only remains to simplify the coefficients in front of the $q$-Hahn polynomials in the previous equation. Without showing the details of the computations, let us indicate that we used the identities
 \begin{equation}
   (a;q)_k = (-a)^k q^{\frac{k(k-1)}{2}} (q^{1-k}a^{-1};q)_k, \label{identity1}
\end{equation}
\begin{equation}
    (q^{-\ell};q)_k = (-1)^kq^{\frac{k(k-1)}{2}-\ell k} \frac{(q;q)_{\ell}}{(q;q)_{\ell-k}}, \label{identity2}
\end{equation}
as well as the relation
\begin{equation}
    \frac{(\hat{\beta}q^{m+1};q)_{N-n}}{(\hat{\beta}q^{N-n+1};q)_{m}}=
    \frac{(\hat{\beta}q;q)_{N-n}}{(\hat{\beta}q;q)_{m}}.
\end{equation}
\end{proof}

\subsection{Orthogonality relations}
From the orthogonality relations \eqref{eq:orthef} of the vectors solving the EVPs and the resolutions of the identity \eqref{eq:idef}, two orthogonality relations for the overlaps $S_m(n)$ and $\Tilde{S}_m(n)$ directly follow. We now show that these correspond to the orthogonality relations of the $q$-Hahn and dual $q$-Hahn polynomials. 

The first orthogonality relation reads
\begin{equation}
    \sum_{n=0}^N \Tilde{S}_m(n)S_{m'}(n) = \delta_{mm'}, \quad m,m'=0,1,\dots,N. \label{eq:orthS1}
\end{equation}
Substituting in \eqref{eq:orthS1} the expressions for the overlap functions according to equations \eqref{Shahn} and \eqref{sthahn}, and then using the identity
\begin{equation}
(a;q)_{k-\ell}=\frac{(a;q)_k}{(q^{1-k}/a;q)_\ell}\left(-\frac{q}{a}\right)^\ell q^{\frac{1}{2}\ell(\ell-1)-k\ell} \label{eq:idorth}
\end{equation}
as well as the relation
\begin{equation}
    (\hat{\alpha} \hat{\beta}q^{2m+2};q)_{N-m} (\hat{\alpha} \hat{\beta} q^{m+1};q)_m = \frac{(\hat{\alpha}\hat{\beta} q^2;q)_N (\hat{\alpha}\hat{\beta} q^{N+2};q)_m}{(\hat{\alpha}\hat{\beta} q;q)_m} \frac{(1-\hat{\alpha}\hat{\beta}q)}{(1-\hat{\alpha}\hat{\beta}q^{2m+1})}, \label{simplif}
\end{equation}
one arrives at
\begin{align}
    &\sum_{n=0}^N \frac{(\hat{\alpha}q,q^{-N};q)_n }{(q,\hat{\beta}^{-1}q^{-N};q)_n }(\hat{\alpha} \hat{\beta} q)^{-n} Q_m (q^{-n};\hat{\alpha},\hat{\beta},N) Q_{m'} (q^{-n};\hat{\alpha},\hat{\beta},N)\nonumber\\
    & \qquad = \frac{(\hat{\alpha} \hat{\beta} q^2;q)_N}{(\hat{\beta}q;q)_N (\hat{\alpha}q)^N} \frac{(q,\hat{\alpha} \hat{\beta} q^{N+2},\hat{\beta}q;q)_m }{(\hat{\alpha}q,\hat{\alpha} \hat{\beta}q,q^{-N};q)_m } \frac{(1-\hat{\alpha} \hat{\beta} q)(-\hat{\alpha}q)^m}{(1-\hat{\alpha} \hat{\beta} q^{2m+1})}q^{\frac{1}{2}m(m-1)-Nm} \delta_{m{m'}}. \label{eq:orthHahn}
\end{align}
Equation \eqref{eq:orthHahn} corresponds to the orthogonality relation of the $q$-Hahn polynomials given in \cite{Koekoek}.

The second orthogonality relation for the overlap functions is
\begin{equation}
     \sum_{m=0}^N \Tilde{S}_m(n)S_{m}({n'}) = \delta_{n{n'}}. \label{eq:orthS2}
\end{equation}
Proceeding similarly as above, again with the help of \eqref{eq:idorth} and \eqref{simplif}, one can rewrite equation \eqref{eq:orthS2} as
\begin{align}
    &\sum_{m=0}^N \frac{(\hat{\alpha}q,\hat{\alpha} \hat{\beta}q,q^{-N};q)_m}{(q,\hat{\alpha} \hat{\beta} q^{N+2},\hat{\beta}q;q)_m } \frac{(1-\hat{\alpha} \hat{\beta} q^{2m+1})}{(1-\hat{\alpha} \hat{\beta} q)(-\hat{\alpha}q)^m} q^{Nm- \frac{1}{2}m(m-1)} Q_m (q^{-n};\hat{\alpha},\hat{\beta},N;q) Q_{m} (q^{-n'};\hat{\alpha},\hat{\beta},N;q)\nonumber\\
    &\qquad = \frac{(\hat{\alpha} \hat{\beta} q^2;q)_N}{(\hat{\beta}q;q)_N}(\hat{\alpha}q)^{-N} \frac{(q,\hat{\beta}^{-1} q^{-N};q)_n }{(\hat{\alpha}q,q^{-N};q)_n }(\hat{\alpha} \hat{\beta} q)^{n} \delta_{n{n'}}. \label{eq:orthdHahn}
\end{align}
Using the duality relation \eqref{eq:HahndualHahn}, equation \eqref{eq:orthdHahn} is seen to correspond to the orthogonality relation of the dual $q$-Hahn polynomials given in \cite{Koekoek}.

\subsection{Bispectral properties}

We now show explicitly how the bispectral properties of the $q$-Hahn polynomials, that is, the recurrence relation and difference equation, follow from the actions of the operators $V$, $W=X-\Q{\mu}Z$ and their transposes on the EVP bases. The recurrence relation (resp. difference equation) of the dual $q$-Hahn polynomials can be obtained straightforwardly from the difference equation (resp. recurrence relation) of the $q$-Hahn polynomials by using the duality relation \eqref{eq:HahndualHahn}. Let us mention that the derivation of the bispectral properties of the OPs of $q$-Hahn type is rather standard from the perspective of Leonard pairs, but we record it here within the framework of the meta $q$-Hahn algebra and its representation theory.

\subsubsection{Recurrence relation} 
The recurrence relation for $S_m(n)$ is found by recalling, on the one hand, the diagonal action of the transposed operator $W^\top$ on its eigenbasis $f^\star$, and on the other hand, the tridiagonal action of the operator $W$ on the eigenbasis $e$ of $V$, which follows from the tridiagonal actions of $Z$ and $X$ given in \eqref{eq:actZe} and \eqref{eq:actXe}: 
\begin{align}
    \braket{e_m \mid (W^{\top} \mid f_n^*}) &=(\braket{e_m \mid W) \mid f_n^*} \\
    \rho_n \braket{e_m \mid f_n^*}&=W_{m+1,m}^{(e)}\braket{e_{m+1} \mid f_n^*}+W_{m,m}^{(e)}\braket{e_{m} \mid f_n^*}+W_{m-1,m}^{(e)}\braket{e_{m-1} \mid f_n^*} \\
    \rho_n S_m(n)&=W_{m+1,m}^{(e)}S_{m+1}(n)+W_{m,m}^{(e)}S_{m}(n)+W_{m-1,m}^{(e)}S_{m-1}(n), \label{eq:recS}
\end{align}
where the definition $S_m(n):=\braket{e_m \mid f_n^*}$ is used in the last line.
Substituting in \eqref{eq:recS} the expression \eqref{eq:rhon} for the eigenvalues $\rho_n$ as well as the expression \eqref{Shahn} for the overlaps $S_m(n)$, and using the notation $p_m(n):=Q_m (q^{-n}; \hat{\alpha}, \hat{\beta}, N;q)$, one finds
\begin{align}
     (-\Q{\alpha-n}+\Q{\mu}) p_m(n) =
&q^{\alpha-\beta-N+2m} \frac{(1-q^{N-m})(1-\hat{\alpha} \hat{\beta} q^{m+1})}{a_m (1-\hat{\alpha} \hat{\beta} q^{2m+1})(1-\hat{\alpha} \hat{\beta} q^{2m+2})} W_{m+1,m}^{(e)} p_{m+1}(n) \nonumber \\
+&W_{m,m}^{(e)} p_m(n)\nonumber \\
+&q^{\beta-\alpha+N-2m+2}\frac{a_{m-1}(1-\hat{\alpha} \hat{\beta} q^{2m-1})(1- \hat{\alpha} \hat{\beta} q^{2m})}{(1-q^{N-m+1})(1-\hat{\alpha} \hat{\beta} q^{m})} W_{m-1,m}^{(e)} p_{m-1}(n).   \label{recurQ} 
\end{align}
The matrix entries $W_{i,j}^{(e)}$ are computed from \eqref{action:Zone:coe1}--\eqref{action:Zone:coe3} and \eqref{action:Xone:coe1}--\eqref{action:Xone:coe3}, and can be written as:
\begin{align}
&W_{m+1,m}^{(e)}= q^{\mu} \Q{\beta-m-\mu}a_m,\\
&W_{m,m}^{(e)}=\frac{q^{\alpha-m}\Q{m}\Q{2 \beta -m-N+1} \Q{\beta +\mu -m+1} }{\Q{2 \beta -2 m+1}\Q{2 \beta -2 m+2}} \nonumber \\
&\qquad \qquad - \frac{q^{\alpha-m}\Q{m-N} \Q{2 \beta -m+1} \Q{m+\mu-\beta}}{\Q{2m-2 \beta}\Q{2 \beta -2m+1}}+q^{\alpha } \Q{\mu -\alpha}, \\
&W_{m-1,m}^{(e)}=
q^{2 \alpha-2 \beta+\mu-1}\frac{\Q{m} \Q{m-N-1} \Q{2\beta-m+2} \Q{2 \beta-m-N+1} \Q{m-\beta-\mu-1} }{a_{m-1}(\Q{2m-2 \beta-2})^2\Q{2\beta-2m+1}\Q{2 \beta -2m+3}}.
\end{align}
Substituting these expressions in \eqref{recurQ}, multiplying both sides of the equation by $(1-q)q^{-\alpha}$, then subtracting $(1-q^{\mu-\alpha})p_m(n)$ on both sides, and finally writing everything in terms of the parameters $\hat{\alpha}$ and $\hat{\beta}$ given in \eqref{newpar}, one recovers the recurrence relation of the $q$-Hahn polynomials presented in the standard form \cite{Koekoek} (for $m,n=0,1,\dots,N$)
\begin{equation}
    -(1-q^{-n})p_m(n)=A_m p_{m+1}(n)-(A_m + C_m)p_m(n)+C_mp_{m-1}(n),
\end{equation}
where
\begin{align}
    &A_m=\frac{(1-\hat{\alpha} \hat{\beta} q^{m+1})(1-\hat{\alpha} q^{m+1})(1-q^{m-N})}{(1-\hat{\alpha} \hat{\beta} q^{2m+1})(1- \hat{\alpha} \hat{\beta} q^{2m+2})},\\
    &C_m=-\hat{\alpha} q^{m-N} \frac{(1-q^m) (1-\hat{\alpha} \hat{\beta} q^{m+N+1})(1-\hat{\beta} q^m)}{(1-\hat{\alpha} \hat{\beta} q^{2m})(1-\hat{\alpha} \hat{\beta} q^{2m+1})}.
\end{align}


\subsubsection{Difference equation} 
The difference equation for $S_m(n)$ is found by recalling that the operator $V$ is diagonal in its eigenbasis $e$ while the transpose $V^\top$ is tridiagonal in the eigenbasis $f^*$ of $W^\top$: 
\begin{align}
    (\braket{e_m \mid V) \mid f_n^*} &=\braket{e_m \mid (V^\top \mid f_n^*}) \\
    \nu_m \braket{e_m \mid f_n^*}&=V_{n+1,n}^{\top(f^*)}\braket{e_{m} \mid f_{n+1}^*}+V_{n,n}^{\top(f^*)}\braket{e_{m} \mid f_n^*}+V_{n-1,n}^{\top(f^*)}\braket{e_{m} \mid f_{n-1}^*} \\
    \nu_m S_m(n)&=V^{(f)}_{n,n+1} S_m(n+1) +V^{(f)}_{n,n} S_m(n)+V^{(f)}_{n,n-1} S_m(n-1). \label{eq:diffS}
\end{align}
Substituting the expression \eqref{Shahn} for $S_m(n)$ in \eqref{eq:diffS} yields
\begin{align}
\nu_m p_m(n) =& 
- q^{\beta-\alpha+n+1}\frac{a_n(1-\hat{\alpha} q^{n+1})}{(1-q^{n+1})} V^{(f)}_{n,n+1} p_m(n+1) \nonumber \\ 
&+V^{(f)}_{n,n} p_m(n) \nonumber \\
&-q^{\alpha-\beta-n} \frac{(1-q^n)}{a_{n-1} (1-\hat{\alpha} q^{n})} V^{(f)}_{n,n-1} p_m(n-1).\label{difQ}
\end{align}
Using \eqref{vf-1} and \eqref{vf+1} for the matrix elements $V^{(f)}_{n,n+1}$ and $V^{(f)}_{n,n-1}$, as well as the reparametrization given in \eqref{newpar}, one can rewrite \eqref{difQ} as 
\begin{align}
-\nu_m q^{-\beta-1} (1-q)^2 \ p_m(n) =& 
(1-\hat{\alpha}q^{n+1}) (1-q^{n-N}) p_m(n+1) \nonumber \\ 
&- q^{-\beta-1} (1-q)^2 V^{(f)}_{n,n} \ p_m(n) \nonumber \\
&+\hat{\alpha}q(1-q^n)(\hat{\beta}-q^{n-N-1}) p_m(n-1).\label{difQ2}
\end{align}
Finally, using $-\nu_m q^{-\beta-1} (1-q)^2=q^{-m}(1-q^m)(1-\hat{\alpha} \hat{\beta} q^{m+1})+q^{-2\beta-1} (1-q^\beta)(1-q^{\beta+1})$ and equation \eqref{vf} for the matrix element $V^{(f)}_{n,n}$, one recovers from \eqref{difQ2} the difference equation of the $q$-Hahn polynomials in the standard form \cite{Koekoek} (for $m,n=0,1,\dots,N$)
\begin{equation}
    q^{-m}(1-q^m)(1-\hat{\alpha} \hat{\beta} q^{m+1}) p_m(n) = B(n) p_m(n+1) - \big[B(n) + D(n)\big] p_m(n) + D(n) p_m(n-1),
\end{equation}
where
\begin{align}
    & B(n) = (1-\hat{\alpha}q^{n+1}) (1-q^{n-N}),\\
    & D(n) = \hat{\alpha}q(1-q^n)(\hat{\beta}-q^{n-N-1}).
\end{align}

Let us remark that the recurrence relation and difference equation of the $q$-Hahn polynomials could have been obtained in a similar fashion starting from the functions $\Tilde{S}_m(n)$ instead.  

\section{$q$-Hahn rational functions} \label{sec:qHahnBRF}
This section focuses on the rational functions in the variable $q^{-x}$:
\begin{align}
&\cU_m(x;a,b,N;q)
=q^{m(N-m-b)}\frac{(q^{-N};q)_m}{(q^{-m-b};q)_m}
\qhg{3}{2}\qargu{q^{-m},q^{-x},q^{b+m-N}}{q^{-N},q^{a-x}}{q}{q}, \label{cU}
\\
&\cV_m(x;a,b,N;q)
=\cU_m(N-x;b-a+2,b,N;q^{-1}).\label{cV}
\end{align}
These are called biorthogonal rational functions of $q$-Hahn type since they are biorthogonal partners and can be expressed in terms of $q$-hypergeometric series of type ${}_3\phi_2$, as for the $q$-Hahn polynomials. Their biorthogonality and generalized bispectral properties have been studied previously in \cite{BGVZ} with the help of a triplet of $q$-difference operators from which the meta $q$-Hahn algebra was introduced. Let us point out that the expressions given in \eqref{cU} and \eqref{cV} differ from those given in \cite{BGVZ} by the transformation $q \to q^{-1}$.   
The normalization is such that
\begin{equation}
\lim_{x\to\infty} \cU_m(x;a,b,N;q) = 1 \quad \text{when }|q|> 1,
\end{equation}
as can be verified with the $q$-Chu-Vandermonde summation formula (see equation (II.6) in \cite{GR}). In the limit $q\to 1$, one recovers the biorthogonal rational functions of Hahn type studied in \cite{TVZ21,VZ_Hahn,TVZ24}.

Here, the functions $\cU_m(x):=\cU_m(x;a,b,N;q)$ and $\cV_m(x):=\cV_m(x;a,b,N;q)$ will be interpreted and characterized within the bidiagonal representation of the meta $q$-Hahn algebra $\mqh$.

\subsection{Representation theoretic interpretation}
In the case of the BRFs of $q$-Hahn type, the relevant scalar products to consider are $U_m(n):=\braket{e_m \mid d_n^*}$ and $\tilde U_m(n):= \braket{e_m^* \mid Z \mid d_n}$, that is, overlap functions between EVP and GEVP bases in the representation of $\mqh$. This is made explicit next. 
\begin{prop}
The functions $U_m(n)$ and $\tilde U_m(n)$ are respectively given as follows in terms of the rational $q$-Hahn functions $\cU_m(x;a,b,N;q)$ and $\cV_m(x;a,b,N;q)$, for $m,n=0,1,\dots,N$:
\begin{align}
  U_m(n)&=  
\frac{a_0\dots a_{n-1}}{a_0\dots a_{m-1}}q^{(a+m-N-1)m}
\frac{(q^{1-a};q)_n(q^{1+b};q)_m}{(q;q)_n(q^{m+b-N};q)_m}\,\cU_m(n;a,b,N;q),  \label{Ucu}\\
\tilde U_m(n) &= 
-\frac{a_n\dots a_{N-1}}{a_m\dots a_{N-1}}q^{(a-b-m-1)(N-m)+(b-a+1)(N-n)}\frac{(q^{m+1};q)_{N-m}(q^{a-b-1};q)_{N-n}}{(q;q)_{N-n}(q^{-N};q)_m(q^{-b};q)_{N-2m}}\,\cV_m(n;a,b,N;q), \label{UTcv}
\end{align}
where
\begin{equation}
    a=\alpha-\beta, \qquad b=N-2\beta-1. \label{parab}
\end{equation}
\end{prop}

\begin{proof}
From the expressions \eqref{eq:en} and \eqref{eq:dn*} of the vectors $\ket{e_n}$ and $\ket{d_n^*}$, it is straightforward to compute
\begin{align}
U_m(n)=\braket{e_m \mid d_n^*}&=   \dfrac{a_{0}a_1 \dots a_{n-1}}{a_0 a_1 \dots a_{m-1}}\dfrac{q^{(\alpha-\beta-1)m} q^{(\beta-\alpha)n} (q^{-n+\alpha-\beta};q)_n (q,q^{-N};q)_m}{(q^{-n};q)_n(q^{-m},q^{m-2\beta-1};q)_m} \nonumber \\ 
&\quad \times \sum_{\ell=0}^{N}
    \frac{(q^{-n},q^{-m}, q^{m-2\beta-1};q)_{\ell}}{(q^{-N},q^{-n+\alpha-\beta};q)_{\ell} } \frac{q^\ell}{(q;q)_\ell}. \label{eq:emdn*}
\end{align}
Comparing the sum on the RHS of \eqref{eq:emdn*} with the definition \eqref{cU} of the function $\cU_m(x)$, and performing some simplifications with the help of the identity \eqref{identity1}, one arrives at the formula \eqref{Ucu} with the identification \eqref{parab} for the parameters $a,b$.

Consider now the function $\Tilde{U}_m(n)=\braket{e_m^* \mid Z \mid d_n}$. From the expressions \eqref{eq:en*} and \eqref{Zdn} for the vectors $\ket{e_n^*}$ and $Z\ket{d_n}$, and after a change of the summation index, one finds
\begin{align}
 \tilde U_m(n) 
 &= -
    \dfrac{a_{n}a_{n+1}\dots a_{N-1}}{a_{m}a_{m+1}\dots a_{N-1}} q^{(\alpha+\beta-N)(N-m)+n-N} 
     \frac{(q^{n-N-\alpha+\beta+2};q)_{N-n}(q,q^{-N};q)_{N-m}}{(q^{n-N};q)_{N-n}(q^{m-N},q^{-N-m+2\beta+1};q)_{N-m}}
 \nonumber\\
     &\quad \times \sum_{k=0}^{N}
     \frac{(q^{m-N},q^{n-N},q^{-N-m+2\beta+1};q)_{k}}{(q^{-N},q^{n-N-\alpha+\beta+2};q)_{k} }\frac{(q^{N-\alpha-\beta+1})^{k}}{(q;q)_{k}}. \label{eq:Ut1}
\end{align}
With the parameters $a,b$ given in \eqref{parab}, equation \eqref{eq:Ut1} can be rewritten as
\begin{align}
 \tilde U_m(n) &= -
 \dfrac{a_{n}a_{n+1}\dots a_{N-1}}{a_{m}a_{m+1}\dots a_{N-1}} q^{(a-b-1)(N-m)+n-N} 
     \frac{(q^{n-N-a+2};q)_{N-n}(q,q^{-N};q)_{N-m}}{(q^{n-N};q)_{N-n}(q^{m-N},q^{-m-b};q)_{N-m}}
 \nonumber\\
     &\quad \times \qhg{3}{2}\qargu{q^{m-N}, q^{n-N}, q^{-m-b}}{q^{-N}, q^{n-N-a+2}}{q}{q^{b-a+2}}. \label{eq:Ut2}
\end{align}
Applying \eqref{TF1} twice to \eqref{eq:Ut2} yields
\begin{align}
 \tilde U_m(n) &= -
 \dfrac{a_{n}a_{n+1}\dots a_{N-1}}{a_{m}a_{m+1}\dots a_{N-1}} q^{(a-b-1)(N-m)+n-N} 
     \frac{(q,q^{-N};q)_{N-m}}{(q^{m-N},q^{-m-b};q)_{N-m}}
 \nonumber\\
     &\quad \times \frac{(q^{m+n-N+b-a+2};q)_{N-n}(q^{n-N+b-a+2};q)_{m}}{(q^{n-N};q)_{N-n}(q^{b-a+2};q)_{m}}\qhg{3}{2}\qargu{q^{-m}, q^{n-N}, q^{m+b-N}}{q^{-N}, q^{n-N+b-a+2}}{q}{q^{2-a}}.
\end{align}
Here, one can use identity \eqref{identity1} as well as
\begin{align}
 &(aq^{-\ell};q)_k=\frac{(a;q)_k(qa^{-1};q)_\ell}{(q^{1-k}a^{-1};q)_\ell}q^{-\ell k},\\
 &(a;q)_{\ell+k}=(a;q)_\ell(aq^\ell;q)_k,
 \end{align}
to show that
\begin{align}
    \tilde U_m(n) &= 
-\frac{a_na_{n+1}\dots a_{N-1}}{a_ma_{m+1}\dots a_{N-1}}q^{(a-b-m-1)(N-m)+(b-a+1)(N-n)}\frac{(q^{m+1};q)_{N-m}(q^{a-b-1};q)_{N-n}}{(q;q)_{N-n}(q^{-m-b};q)_m(q^{-b};q)_{N-2m}} \nonumber\\
&\quad \times \qhg{3}{2}\qargu{q^{-m}, q^{n-N}, q^{m+b-N}}{q^{-N}, q^{n-N+b-a+2}}{q}{q^{2-a}}. \label{eq:Ut3}
\end{align}
Finally, from the definitions \eqref{cU} and \eqref{cV}, and the identity
\begin{equation}
(a;q^{-1})_k=(a^{-1};q)_k(-a)^kq^{-\frac{k(k-1)}{2}},
\end{equation}
one can replace 
\begin{equation}
    \qhg{3}{2}\qargu{q^{-m}, q^{n-N}, q^{m+b-N}}{q^{-N}, q^{n-N+b-a+2}}{q}{q^{2-a}} = \frac{(q^{-m-b};q)_m}{(q^{-N};q)_m}\cV_m(n, a, b, N;q) \label{VV}
\end{equation}
in \eqref{eq:Ut3}, which leads to the result \eqref{UTcv}.

\end{proof}

\subsection{Biorthogonality}
It was already mentioned that the rational functions of $q$-Hahn type are biorthogonal. In the following proposition, the biorthogonality relations are provided and proved from the perspective of the meta $q$-Hahn algebra.  
\begin{prop}
The rational functions of $q$-Hahn type $\cU_m(n;a,b,N;q)$ and $\cV_m(n;a,b,N;q)$ satisfy the following biorthogonality relations:
\begin{align}
&\sum_{n=0}^N  
\cW(n)\,\cV_m(n;a,b,N;q)\,
\cU_{m'}(n;a,b,N;q) = h_m \delta_{mm'}, \quad m,m'=0,1,\dots,N, \label{1storth}\\
&\sum_{m=0}^N  
\cW^*(m)\,\cV_m(n;a, b,N;q)
\,\cU_m(n';a, b,N;q) = h_n^*\delta_{nn'}, \quad n,n'=0,1,\dots,N, \label{2ndorth}   
\end{align}
where
\begin{align}
& h_m = q^{(b+1)m} \frac{(q,q^{-N},\,q^{m+b-N};q)_m (q^{2m+b-N+1};q)_{N-2m}}{(q^{b+1};q)_m (q^{b-N+1};q)_N},
\\
& \cW(n) = q^{(a-b-1)n} \frac{(q^{a-b-1};q)_{N-n} (q^{1-a};q)_{n}}{(q^{-b};q)_{N}} \frac{(q;q)_N}{(q;q)_{N-n}(q;q)_{n}},\\
&h_n^*=q^{-n} \frac{(q,q^{2-a+b-N};q)_n}{(q^{-N},q^{1-a};q)_n},\\
& \cW^*(m) = q^{(a-1)N-(b+1)m}
\frac{(q^{2-a+b-N};q)_N(q^{b+1};q)_m}{(q,q^{-N},q^{m+b-N};q)_m(q^{2m+b-N+1}q)_{N-2m}}.
\end{align}   
\end{prop}

\begin{proof}
From the orthogonality relations in \eqref{eq:orthef} and \eqref{eq:orthd}, and the resolutions of the identity in \eqref{eq:idef} and \eqref{eq:idd}, one can show that the overlap functions $U_m(n)$ and $\Tilde{U}_m(n)$ obey the two following biorthogonality relations:
 \begin{align}
     &\sum_{n=0}^N (-1)\Tilde{U}_m(n) U_{m'}(n) = \delta _{mm'}, \quad m,m'=0,1,\dots,N, \label{1stbiorth}\\
    &\sum_{m=0}^N (-1)\Tilde{U}_m(n) U_m(n')  =  \delta _{nn'}, \quad n,n'=0,1,\dots,N. \label{2ndbiorth}
 \end{align}
Replacing the functions $U_m(n)$ and $\tilde U_m(n)$ in \eqref{1stbiorth} and \eqref{2ndbiorth} with their expressions \eqref{Ucu} and \eqref{UTcv}, one gets two biorthogonality relations for the functions $\cU_m(n;a,b,N;q)$ and $\cV_m(n;a,b,N;q)$ which, with some manipulations, lead to the results \eqref{1storth} and \eqref{2ndorth}. Let us indicate that in order to arrive at the presentation \eqref{1storth} from \eqref{1stbiorth}, both sides of the equation have been multiplied by $(q;q)_N/(q^{-b};q)_N$, and the identities \eqref{identity1} and
\begin{equation}
    \frac{(q;q)_N}{(q^{m+1};q)_{N-m}}=(q;q)_m \label{simpleid2}
\end{equation}
have been used. Similarly, to arrive at \eqref{2ndorth} from \eqref{2ndbiorth}, both sides of the equation have been multiplied by $(q^{2-a+b-N};q)_N/(q;q)_N$ and the identities \eqref{identity1}, \eqref{simpleid2} as well as 
\begin{equation}
    \frac{(q^{2-a+b-N};q)_N}{(q^{2-a+b-N+n};q)_{N-n}}=(q^{2-a+b-N};q)_n
\end{equation}
and
\begin{equation}
    \frac{(q;q)_{N-n}}{(q;q)_{N}}=\frac{(-1)^n q^{\frac{n}{2}(n-2N-1)}}{(q^{-N};q)_{n}}
\end{equation}
have been used. With these presentations, the weight functions $\cW(n)$ and $\cW^*(m)$ and the normalizations $h_m$ and $h_n^\star$ are such that $h_0=h_0^*=1$.
\end{proof}

\begin{rem}
Relation \eqref{1storth} corresponds to the biorthogonality relation provided in \cite{BGVZ}, with the change $q\to q^{-1}$. The weight function $\cW(n)$ and normalization $h_{m}$ are found naturally in the present framework using the representation theory of the meta $q$-Hahn algebra, while with the methods used in \cite{BGVZ}, they required hindsight and comparison with the limit of a biorthogonality relation provided by Wilson \cite{Wil}.    
\end{rem}

\subsection{Bispectrality of $\cU_m(n)$}
The rational functions $\cU_m(n)$ satisfy generalized bispectral properties. These are recovered next with the help of the actions of the generators of $\mqh$ on EVP and GEVP bases.

\subsubsection{Recurrence relation}
\begin{prop}
The rational functions of $q$-Hahn type $\cU_m(n; a, b, N;q)$ defined in \eqref{cU} obey the following recurrence relation (for $m,n=0,1,\dots,N$):
\begin{align}
& \Q{n-m-a}\cA_m (\cU_{m+1}(n)-\cU_{m}(n)) +\Q{n+m-a+b-N}\cC_m (\cU_{m-1}(n)-\cU_{m}(n))\nonumber\\
&\qquad =\Q{a}\Q{2m+b-N} \,\cU_{m}(n),
\label{cU:RRalt}
\end{align}
where
\begin{align}
   &\cA_m=q^{a+m}\frac{\Q{-m-b-1}\Q{m+b-N}}{\Q{N-2m-b-1}}, \label{cA}\\  
   &\cC_m=q^{a}\frac{\Q{m}\Q{m-N-1}}{\Q{2m+b-N-1}}. \label{cC}
\end{align}
\end{prop}

\begin{proof}
Since the vectors $\ket{d_n^\star}$ are solutions to the GEVP between $X^\top$ and $Z^\top$, one has 
\begin{align}
    (\braket{e_m \mid X) \mid d_n^*} &=\braket{e_m \mid X^{\top}\mid d_n^*} \nonumber\\
    &=\braket{e_m \mid \lambda_n Z^{\top} \mid d_n^*} \nonumber\\
    &=\lambda_n (\braket{e_m \mid Z) \mid d_n^*}, \label{eq:recUmn1}
\end{align}   
where the generalized eigenvalues $\lambda_n = \Q{\alpha-n}$ were given in \eqref{eq:lambda}.
With the tridiagonal actions \eqref{eq:actZe} and \eqref{eq:actXe} of $Z$ and $X$ on the $e$ basis, equation \eqref{eq:recUmn1} leads to a recurrence relation for the overlaps $U_m(n)=\braket{e_m \mid d_n^\star}$:
\begin{align}
&X^{(e)}_{m+1,m}U_{m+1}(n) +X^{(e)}_{m,m} U_{m}(n)+X^{(e)}_{m-1,m} U_{m-1}(n)
\nonumber\\
&=\lambda_n \left(Z^{(e)}_{m+1,m}U_{m+1}(n) +Z^{(e)}_{m,m} U_{m}(n)+Z^{(e)}_{m-1,m} U_{m-1}(n)\right). \label{recU}
\end{align}
The terms in the previous equation can be rearranged as
\begin{align}
&(X^{(e)}_{m+1,m}-\lambda_n Z^{(e)}_{m+1,m}) U_{m+1}(n)+(X^{(e)}_{m-1,m}-\lambda_n Z^{(e)}_{m-1,m}) U_{m-1}(n)
\nonumber\\
&=(\lambda_n Z^{(e)}_{m,m}-X^{(e)}_{m,m}) U_{m}(n). \label{eq:recUmn2}
\end{align}
Now, the following can be substituted in \eqref{eq:recUmn2}:  the matrix elements \eqref{action:Zone:coe1}, \eqref{action:Zone:coe3}, \eqref{action:Xone:coe1} and \eqref{action:Xone:coe3}, and the expression \eqref{Ucu} of $U_m(n)$ in terms of the functions $\cU_m(n)$. Bringing in the parameters \eqref{parab}, and after some rearrangement, equation \eqref{eq:recUmn2} then gives
\begin{align}
&\Q{n-m-a} \cA_m (\cU_{m+1}(n)-\cU_{m}(n))+\Q{n+m-a+b-N}\cC_m (\cU_{m-1}(n)-\cU_{m}(n)) \nonumber \\
&=\Big(q^{\frac{1}{2}(2n-2a+b-N+1)}\Q{2m+b-N}(\lambda_n Z^{(e)}_{m,m}-X^{(e)}_{m,m}) \nonumber \\
& \qquad \qquad \qquad-\Q{n-m-a} \cA_m-\Q{n+m-a+b-N}\cC_m \Big) \cU_{m}(n),
\end{align}
where the recurrence coefficients $\cA_m$ and $\cC_m$ are as in \eqref{cA} and \eqref{cC}. 
Finally, using the matrix elements \eqref{action:Xone:coe2} and \eqref{action:Zone:coe2} and rewriting them in terms of the parameters \eqref{parab}, one finds \eqref{cU:RRalt}.
\end{proof}

\subsubsection{Difference equation}
\begin{prop} \label{diifcU}
The rational functions of $q$-Hahn type $\cU_m(n; a, b, N;q)$ defined in \eqref{cU} obey the following difference equation (for $m,n=0,1,\dots,N$):
\begin{align}
 &   \cB_n \,\cU_{m}(n+1)
 - (\cB_n+\cD_{n}) \,\cU_{m}(n)
    + \cD_{n} \,\cU_{m}(n-1)\nonumber\\
& \qquad = \Q{m} \Q{N-m-b} \Big(\Q{a-n} \,\cU_{m}(n)-q^{a}\Q{-n}\,\cU_{m}(n-1)\Big), \label{diffeqcU}
\end{align}
where
\begin{align}
  &\cB_n=q^{n-b}\Q{a-n}\Q{a-n-1}\Q{N-n},\label{diffeqcU_coecB}\\
   &\cD_{n}=q^{n}\Q{-n}\Q{a-n}\Q{N-n+a-b}. \label{diffeqcU_coecD}
\end{align}
\end{prop}

\begin{proof}
Recalling that $\ket{e_n}$ are eigenvectors for $V$ with eigenvalues $\nu_n$ given by \eqref{mu}, one can write
\begin{align}
\braket{e_m \mid V^{\top}Z^{\top}\mid d_n^*} &=(\braket{e_m \mid V) Z^{\top} \mid d_n^*} \nonumber\\
&=\nu_m\braket{e_m \mid Z^{\top}\mid d_n^*}. \label{eq:difUmn1}
\end{align}
The action of the operator $V^\top Z^\top$ on the basis $d^*$ is tridiagonal, as given in the formula \eqref{action:VZTond}, and the action of $Z^\top$ is bidiagonal, as given in the formula \eqref{ZTond*}. Using these in \eqref{eq:difUmn1} leads to the following difference equation for the overlaps $U_m(n)=\braket{e_m \mid d_n^*}$: 
\begin{align} &(V^{\top}Z^{\top})_{n+1,n}^{(d*)} U_{m}(n+1)+(V^{\top}Z^{\top})_{n,n}^{(d*)}U_{m}(n)+(V^{\top}Z^{\top})_{n-1,n}^{(d*)}U_m(n-1)\nonumber\\
 & =\nu_m (-U_m(n)+a_{n-1} U_m(n-1)) \label{eq:difUmn2}.
\end{align}
The expression \eqref{Ucu} for $U_m(n)$ in terms of $\cU_m(n)$ can be substituted in equation \eqref{eq:difUmn2}, which becomes
\begin{align} &(V^{\top}Z^{\top})_{n+1,n}^{(d*)} \frac{a_n q^{-a}(1-q^{a-n-1})}{(1-q^{-n-1})} \cU_{m}(n+1)\nonumber \\
&+(V^{\top}Z^{\top})_{n,n}^{(d*)}\cU_{m}(n)+(V^{\top}Z^{\top})_{n-1,n}^{(d*)}\frac{q^a(1-q^{-n})}{a_{n-1}(1-q^{a-n})}\cU_m(n-1)\nonumber\\
& =\nu_m (-\cU_m(n)+\frac{q^a(1-q^{-n})}{(1-q^{a-n})}\cU_m(n-1)). \label{eq:difUmn3}
\end{align}
Now, one can substitute in \eqref{eq:difUmn3} the matrix elements \eqref{VZtop1}--\eqref{VZtop3} expressed in terms of the parameters $a,b$ given by \eqref{parab}, and multiply both sides of the equation by $-\frac{q^{\frac{1}{2}(1+b-N)}}{\Q{a-n}}$ to obtain:
\begin{align} 
&\cB_n \cU_{m}(n+1)
-\left(\cB_n+ \cD_n+\Q{a-n} \Q{(N-b-1)/2}\Q{(N-b+1)/2} \right) \cU_{m}(n) \nonumber \\
&\quad+\left(\cD_n+q^a \Q{-n} \Q{(N-b-1)/2}\Q{(N-b+1)/2}\right) \cU_m(n-1)\nonumber\\
&\qquad =q^{\frac{1}{2}(N-b-1)}\nu_m(\Q{a-n}\cU_m(n)-q^a \Q{-n} \cU_m(n-1)), \label{eq:difUmn4}
\end{align}
where the coefficients $\cB_n$ and $\cD_n$ are those given in \eqref{diffeqcU_coecB} and \eqref{diffeqcU_coecD}. 
Finally, using the eigenvalue \eqref{mu} in \eqref{eq:difUmn4} and observing that
\begin{equation}
    q^{\frac{1}{2}(N-b-1)}\Q{m-(N-b-1)/2}\Q{(N-b-1)/2-m+1}+\Q{(N-b-1)/2}\Q{(N-b+1)/2}\\
    =\Q{m}\Q{N-b-m},
\end{equation}
one obtains the difference equation given in \eqref{diffeqcU}.
\end{proof}

\begin{rem}
The difference equation \eqref{diffeqcU} for $\cU_m(n)$ corresponds to the one provided in \cite{BGVZ} under the inversion $q\to q^{-1}$. Concerning the recurrence relation \eqref{cU:RRalt}, we have noticed some errors in the equation (7.3) provided in \cite{BGVZ} (in the coefficients $\mu_n^{(2)}$ and $\mu_n^{(8)}$), but otherwise, there is also a correspondence under the same inversion for the parameter $q$.   
\end{rem}

\subsection{Bispectrality of $\cV_m(n)$}
The rational functions $\cV_m(n)$ also satisfy generalized bispectral properties. This is apparent from equation \eqref{cV}, which relates $\cV_m(n;a,b,N;q)$ and $\cU_m(n;a,b,N;q)$ by the transformations
\begin{equation}
n\to N-n, \qquad a \to b-a+2, \qquad q \to q^{-1}. \label{eq:transfvarspar}  
\end{equation}
Applying these transformations to the recurrence relation and difference equation obtained in the previous subsection for the functions $\cU_m(n)$ yields directly the bispectral properties of the functions $\cV_m(n)$. However, since both biorthogonal partners are equally important, we present in this subsection an independent algebraic derivation of the bispectrality of $\cV_m(n)$.

\subsubsection{Recurrence relation}
\begin{prop}
The rational functions of $q$-Hahn type $\cV_m(n; a, b, N;q)$ defined in \eqref{cV} obey the following recurrence relation (for $m,n=0,1,\dots,N$):
\begin{align}
& \Q{n-N+m+b-a+2}q^{2N-4m-2b}\cA_m (\cV_{m+1}(n)-\cV_{m}(n)) \nonumber\\
&+\Q{n-m-a+2}\cC_m \big(\cV_{m-1}(n)-\cV_{m}(n)\big)=-q^2\Q{a-b-2}\Q{N-2m-b} \,\cV_{m}(n),
\label{cV:RR}
\end{align}
with $\cA_m$ and $\cC_m$ given by \eqref{cA} and \eqref{cC}.
\end{prop}
\begin{proof}
The functions $\cV_m(n)$ appear within the overlaps $\tilde U_m(n) = \braket{e_m^* \mid Z\mid d_n}$. Let us define $\ket{\tilde d_n}:= Z \ket{d_n}$, so that $\tilde U_m(n) = \braket{e_m^*\mid \tilde d_n}$. Recalling the GEVP $X\ket{d_n}=\lambda_nZ\ket{d_n}$ and using the first relation \eqref{eq:mqhrel1} of the meta $q$-Hahn algebra, one finds
\begin{align}
   \lambda_n \braket{e_m^* \mid Z \mid \tilde d_n}&=\lambda_n \braket{e_m^* \mid Z^2 \mid d_n} \nonumber\\
   &=\braket{e_m^* \mid ZX \mid d_n} \nonumber\\
   &=\braket{e_m^* \mid (Z^2+Z-(1-q)X+qXZ \mid d_n} \nonumber\\
   &=\braket{e_m^* \mid (Z+1-(1-q)\lambda_n+qX \mid \tilde d_n}. \label{eq:recUtmn1}
\end{align}
With the tridiagonal actions of the transposed operators $Z^\top$ and $X^\top$ on the $e^*$ basis, as deduced from \eqref{eq:actZe} and \eqref{eq:actXe}, equation \eqref{eq:recUtmn1} leads to the following recurrence relation for the overlaps $\tilde U_m(n)$:
\begin{align}
  &(q X_{m,m+1}^{(e)}+Z_{m,m+1}^{(e)}  )\tilde U_{m+1}(n) 
  +(q X_{m,m}^{(e)}+Z_{m,m}^{(e)}+1)  \tilde U_{m}(n) +(q X_{m,m-1}^{(e)} +Z_{m,m-1}^{(e)} ) \tilde U_{m-1}(n) 
\nonumber\\
&=\lambda_n\left(
Z_{m,m+1}^{(e)}  \tilde U_{m+1}(n)
+(Z_{m,m}^{(e)}+1-q)  \tilde U_{m}(n)
+Z_{m,m-1}^{(e)}  \tilde U_{m-1}(n)
\right). \label{eq:recUtmn2}
\end{align}
Then, similarly to the computations performed in the previous subsection, equation \eqref{cV:RR} can be obtained from \eqref{eq:recUtmn2} by writing the overlaps $\tilde U_{m}(n)$ in terms of the functions $\cV_m(n)$ according to \eqref{UTcv}, substituting in the matrix elements \eqref{action:Zone:coe1}--\eqref{action:Zone:coe3} and \eqref{action:Xone:coe1}--\eqref{action:Xone:coe3} as well as the generalized eigenvalue \eqref{eq:lambda}, and using the parameters \eqref{parab}.
\end{proof}
It can be verified that the recurrence relation \eqref{cV:RR} satisfied by $\cV_m(n)$ indeed coincides with the recurrence relation \eqref{cU:RRalt} satisfied by $\cU_m(n)$ after applying the transformations \eqref{eq:transfvarspar}.

\subsubsection{Difference equation}
\begin{prop} \label{diifcV}
The rational functions of $q$-Hahn type $\cV_m(n; a, b, N)$ defined in \eqref{cV} obey the following difference equation (for $m,n=0,1,\dots,N$):  
\begin{align}
 &  \tilde\cB_{n,m} \Big(\cV_m(n+1;a,b,N) - \cV_m(n;a,b,N)\Big)
  + \tilde\cD_{n} \Big(\cV_m(n-1;a,b,N) - \cV_m(n;a,b,N)\Big)
  \nonumber \\
 & \qquad = \Q{-m}\Q{m+b-N}\Q{a-b-2}\cV_m(n;a,b,N),
 \label{diffeqcV}
\end{align}
where
\begin{align}
  &\tilde\cB_{n,m} = q^{n-N}\Q{N-n}\Q{N-n-m+a-b-2}\Q{m-n+a-2},\\
  &\tilde\cD_{n} = -q^{b-N}\Q{n}\Q{N-n+a-b-1}\Q{N-n+a-b-2}.
\end{align}
\end{prop}

\begin{proof}
Using the EVP $V^\top \ket{e_m^*} = \nu_m\ket{e_m^*}$, the GEVP $X\ket{d_n}=\lambda_nZ\ket{d_n}$ as well as the first and third relations \eqref{eq:mqhrel1}, \eqref{eq:mqhrel3} of the meta $q$-Hahn algebra, 
one finds
\begin{align}
    \nu_m \braket{e_m^* \mid Z \mid \tilde d_n} &=
    \braket{e_m^* \mid VZ^2 \mid  d_n} \nonumber\\
    &=\braket{e_m^* \mid (qZV+(1+q)X-(1-q)V+\eta)Z \mid  d_n} \nonumber\\
    &=\braket{e_m^* \mid Z(qVZ+\eta)+(1+q)XZ-(1-q)VZ \mid  d_n} \nonumber\\
    &=\braket{e_m^* \mid Z(qVZ+\eta)+(1+q^{-1})(ZX-Z^2-Z+(1-q)X)-(1-q)VZ \mid  d_n}\nonumber\\
    &=\braket{e_m^* \mid Z\left(qVZ+(1+q^{-1})((\lambda_n-1)Z+\lambda_n(1-q)-1)+\eta-(1-q)\nu_m\right)\mid  d_n}. \label{eq:difVmn1}
\end{align}
From the bidiagonal action \eqref{action:Zond} of $Z$ on $\ket{d_n}$, one can replace on the LHS of \eqref{eq:difVmn1}
\begin{equation}
Z \ket{\tilde d_n} =Z^2 \ket{d_n}= Z(-\ket{d_n} +a_n \ket{d_{n+1}})  = -\ket{\tilde d_n } + a_n\ket{\tilde d_{n+1}}. \label{eq:difVmn2}
\end{equation} 
Using the tridiagonal action \eqref{action:VZond} of $VZ$ on $\ket{d_n}$ in addition to \eqref{action:Zond} on the RHS of \eqref{eq:difVmn1}, and recalling that $\tilde U_{m}(n)=\braket{e_m^* \mid \tilde d_n}$, one arrives at the following difference equation for the overlaps:
\begin{align}
\nu_m\left( -q \tilde U_{m}(n)
+a_n \tilde U_{m}(n+1)
\right) &=
\left(q(VZ)^{(d)}_{n+1,n}+a_n(1+q^{-1})(\lambda_n - 1)\right)\tilde U_{m}(n+1) \nonumber \\
&\quad + \left(q(VZ)^{(d)}_{n,n} - (1+q)\lambda_n+ \eta \right) \tilde U_{m}(n) + q(VZ)^{(d)}_{n-1,n} \tilde U_{m}(n-1). \label{eq:difVmn3}
\end{align}
The relation \eqref{eq:eta3} for the parameter $\eta$ of the algebra $\mqh$ can be used in \eqref{eq:difVmn3}. Substituting the values \eqref{eq:lambda} and \eqref{mu} for $\lambda_n$ and $\nu_n$, the expression \eqref{UTcv} for $\Tilde{U}_m(n)$ in terms of $\cV_m(n)$ and the formulas \eqref{VZtop1}--\eqref{VZcoe} for the matrix elements of $VZ$, and writing everything in terms of the parameters $a,b$ given in \eqref{parab}, one obtains the difference equation \eqref{diffeqcV}.
\end{proof}
Again, let us stress that the difference equation \eqref{diffeqcV} of $\cV_m(n)$ can be recovered from the difference equation \eqref{diffeqcU} of $\cU_m(n)$ by applying the transformations \eqref{eq:transfvarspar}. In fact, by observing that
\begin{equation}
\tilde\cB_{n,m} = \tilde\cB_{n,0}+q^{a-2-N} \Q{m}\Q{N-n}\Q{N-b-m}
\end{equation}
with
\begin{equation}
\tilde\cB_{n,0} = q^{n-N}\Q{N-n}\Q{N-n+a-b-2}\Q{a-n-2},
\end{equation}
one can rewrite the difference equation \eqref{diffeqcV} in the form 
\begin{align}
 & \tilde\cB_{n,0} \,\cV_{m}(n+1)
 - (\tilde\cB_{n,0}+\tilde\cD_{n}) \,\cV_{m}(n)
    + \tilde\cD_{n} \,\cV_{m}(n-1)\nonumber\\
& \qquad = q^{-N}\Q{m} \Q{N-m-b} \left(q^b\Q{a-b-n+1} \,\cV_{m}(n)-q^{a-2}\Q{N-n}\,\cV_{m}(n+1)\right), \label{diffeqcV2}
\end{align}
which makes the similarity with \eqref{diffeqcU} more apparent.

\subsection{Contiguity relations}
We finally provide some contiguity relations for the functions $\cU_m(n;a,b,N;q)$, again using the representation theory of $\mqh$. These contiguity relations involve a shift in the parameter $a$. 
\begin{prop}  
    The rational functions of $q$-Hahn type $\cU_m(n; a, b,N;q)$ defined in \eqref{cU} satisfy the following contiguity relations (for $m,n=0,1,\dots,N$):
 \begin{align}
    &\Q{a} \, \cU_m(n;a+1)  =\Q{a-n} \, \cU_m(n;a)-q^a \Q{-n} \, \cU_m(n-1;a), \label{cont1}
\end{align}
and
\begin{align}
 &-\frac{\Q{a}\, \Q{2m+b-N}}{\Q{a-n}}\,\cU_m(n;a+1) \nonumber\\
&= \cA_m \, \cU_{m+1}(n;a)-(\cA_m+\cC_m+(1-q^a)\Q{2m+b-N})\,\cU_m(n;a)+\cC_m \,\cU_{m-1}(n;a), \label{cont2}   
\end{align}
where $\cU_m(n; a):=\cU_m(n; a, b,N;q)$ for simplicity and the coefficients $\cA_m$ and $\cC_m$ are the same as in \eqref{cA} and \eqref{cC}.
\end{prop}
\begin{proof}
From the expression \eqref{eq:dn*} for $\ket{d_n^*}$ and the action \eqref{eq:actZTn} of $Z^\top$ on the standard basis, one can compute
\begin{align}
    Z^{\top}\ket{d_n^*} = -\sum_{\ell=0}^{N} \dfrac{a_0 a_{1}\dots a_{n-1}}{a_{0} a_{1}\dots a_{\ell-1}} \frac{q^{(\beta-\alpha)n}}{q^{(\beta-\alpha)\ell}} \dfrac{(q^{-n};  q)_{\ell}}{(q^{-n};  q)_{n}} \dfrac{(q^{-n+(\alpha+1)-\beta};  q)_{n}}{(q^{-n+(\alpha+1)-\beta};  q)_{\ell}} \ket{\ell}. \label{eq:Ztdsn}
\end{align}
Then, using the expression \eqref{eq:en} for the vectors $\ket{e_n}$, the result \eqref{eq:Ztdsn} and the orthonormality of the standard basis, it is straightforward to show 
\begin{align}
 U_{m}(n)\Big|_{\alpha\to\alpha+1} 
 =- q^{m-n} \braket{e_m \mid Z^{\top}\mid d_n^*}. \label{eq:cont1}
\end{align}
Now, acting with $Z^\top$ on the vector $\ket{d_n^*}$ in \eqref{eq:cont1} according to the bidiagonal action \eqref{ZTond*}, one obtains 
\begin{align}
   q^{n-m} \, U_{m}(n)\Big|_{\alpha\to\alpha+1}=U_{m}(n)- a_{n-1} U_{m}(n-1). \label{CR:U}
\end{align}
Equation \eqref{CR:U} leads to the contiguity relation \eqref{cont1} after rewriting the overlaps $U_m(n)$ in terms of the functions $\cU_m(n)$ with the help of \eqref{Ucu}. By acting instead with $Z$ on the vector $\ket{e_m}$ in \eqref{eq:cont1} according to the tridiagonal action \eqref{eq:actZe}, recalling that $\braket{e_m \mid Z^{\top}\mid d_n^*}=(\braket{e_m \mid Z)\mid d_n^*}$, one finds 
\begin{align}
 &q^{n-m} \, U_{m}(n)\Big|_{\alpha\to\alpha+1} = -Z^{(e)}_{m+1,m}U_{m+1}(n)-Z^{(e)}_{m,m}U_{m}(n)-Z^{(e)}_{m-1,m}U_{m-1}(n). \label{CR2:U} 
\end{align}
Substituting the matrix elements \eqref{action:Zone:coe1}--\eqref{action:Zone:coe3} for $Z$ in \eqref{CR2:U} and rewriting again the overlaps $U_m(n)$ in terms of $\cU_m(n)$, one arrives at the contiguity relation \eqref{cont2}.
\end{proof}

\section{Conclusion} \label{sec:conclusion}

The program of extending the Askey scheme to biorthogonal rational functions through the introduction of meta algebras and their representation, initiated in \cite{TVZ24} with the Hahn case, was pursued in the present paper by dealing with the $q$-Hahn case. Following the general approach of \cite{TVZ24}, the meta $q$-Hahn algebra with three generators $X,V,Z$ was defined abstractly, and a representation on a finite-dimensional vector space where all three generators act in a bidiagonal fashion was obtained. This bidiagonal representation allowed to explicitly solve generalized and ordinary eigenvalue problems involving the operators $X,V,Z$ or their transpositions. It also proved quite straightforward to compute the actions of operators on the (generalized) eigenbases using the explicit expressions of the vectors. In particular, it was found that the operators $V$ and $W=X-\Q{\mu}Z$ form a Leonard pair \cite{Ter01}, in accordance with the fact that when viewed as abstract algebraic elements, $V$ and $W$ generate the familiar $q$-Hahn algebra. Generally speaking, this is the reason why the representation theory of meta algebras leads to a unified interpretation of both polynomials and rational functions, as the meta algebras contain those of Askey--Wilson type \cite{Zhe91}. In the present case, the $q$-Hahn polynomials were identified within overlaps between eigenbases involving the operators $V$ and $W$, while the rational functions of $q$-Hahn type were identified within overlaps between (generalized) eigenbases involving $V$, $X$ and $Z$. Then, working with the properties of the (generalized) eigenbases and the actions of the operators on these bases, the (bi)orthogonality and bispectral properties of the $q$-Hahn functions, as well as some contiguity relations, were obtained with rather simple computational efforts. 

The results in this paper agree, as they should, with those on biorthogonal rational functions of $q$-Hahn type found in \cite{BGVZ} and obtained in a different way by working with a $q$-difference realization of the meta $q$-Hahn algebra; and they also agree with the classical results on $q$-Hahn and dual $q$-Hahn polynomials \cite{Koekoek}. Moreover, the formulas in the Hahn case \cite{TVZ24} can be recovered by taking the limit $q\to 1$ of those in the $q$-Hahn case presented here, as expected on general grounds.

The next step in this meta algebra program is to study with the same approach the functions of Racah type (${}_4F_3$ hypergeometric series) and $q$-Racah type (${}_4\phi_3$ basic hypergeometric series), which are at the top of the terminating branch of the Askey scheme. In these cases, however, the corresponding meta algebras remain to be identified. Other directions for future research have been discussed at the end of the paper \cite{TVZ24} and can be mentioned here as well: studying infinite-dimensional representations of meta algebras in order to include infinite families of rational functions and polynomials in the picture, or exploring multivariate generalizations of these. Finally, concerning the elliptic BRFs mentioned in the introduction, an open problem is the relation of meta algebras of Askey--Wilson type with elliptic Sklyanin algebras. 

\section*{Acknowledgments}
PAB holds an Alexander-Graham-Bell scholarship from the Natural Sciences and Engineering Research Council of Canada (NSERC). AB, STe and SP each benefited from a CRM-ISM undergraduate summer scholarship. The research of STs is supported by JSPS KAKENHI (Grant Number 24K00528). The research of LV is supported by a Discovery Grant from the NSERC. MZ was funded in part by an Alexander-Graham-Bell scholarship from NSERC and a doctoral scholarship from Fonds de Recherche du Québec - Nature et Technologies (FRQNT), and is currently supported by a postdoctoral fellowship from NSERC and by Perimeter Institute for Theoretical Physics. Research at Perimeter Institute is supported in part by the Government of Canada through the Department of Innovation, Science and Economic Development and by the Province of Ontario through the Ministry of Colleges and Universities.  

\appendix
\section{Actions of operators on several bases} \label{sec:appendix}

This appendix provides some useful formulas for the actions of operators representing elements of $\mqh$ on bases solving GEVPs or EVPs, in addition to the actions already provided in Section \ref{sec:actionbases} (let us recall here that $a_{-1}=a_{N}=0$):
{\allowdisplaybreaks
\begin{align}
&Z \ket{d_n} = -\ket{d_n} + a_n \ket{d_{n+1}},
\label{action:Zond}\\
&X \ket{d_n} = \lambda_n Z \ket{d_n} = -\lambda_n \ket{d_n} + \lambda_n a_n \ket{d_{n+1}}, \label{Xond}\\
&V \ket{d_n} = \sum_{j=\max(0,n-1)}^N V_{j,n}^{(d)}\ket{d_j}, \label{Vond}\\
&Z^{\top} \ket{d_n^*} = -\ket{d_n^*} + a_{n-1} \ket{d_{n-1}^*}, \label{ZTond*}\\
&X^{\top} \ket{d_n^*} = \lambda_n Z^{\top} \ket{ d_n^*}=- \lambda_n \ket{ d_n^* }+  \lambda_n a_{n-1} \ket{d_{n-1}^*}, \label{XTond*} \\
&V^{\top} \ket{d_n^*} = \sum_{j=0}^{\max(n+1,N)} V_{j,n}^{\top(d*)}\ket{d_j^*}, \label{VTond*}\\
&V^{\top}Z^{\top} \ket{d_n^*} = (V^{\top}Z^{\top})_{n+1,n}^{(d*)} \ket{d_{n+1}^*} + (V^{\top}Z^{\top})_{n,n}^{(d*)} \ket{d_{n}^*} + (V^{\top}Z^{\top})_{n-1,n}^{(d*)} \ket{d_{n-1}^*}, \label{action:VZTond}\\
&VZ\ket{d_n} 
=(VZ)_{n+1,n}^{(d)} \ket{d_{n+1}}
+(VZ)_{n,n}^{(d)} \ket{d_{n}}
+(VZ)_{n-1,n}^{(d)} \ket{d_{n-1}},
\label{action:VZond}\\
&Z\ket{f_n} =-\ket{f_n} +\sum_{j=n+1}^N \frac{(-1)^{j+n+1}a_{n}a_{n+1}\cdots a_{j-1}}{q^{\frac{1}{2}(j-n-1)(2\alpha-2\mu-n-j)}}\ket{f_{j}}, \\
&X\ket{f_n} =-\lambda_n \ket{f_n} +[\mu]_q\sum_{j=n+1}^N \frac{(-1)^{j+n+1} a_{n}a_{n+1}\cdots a_{j-1}}{q^{\frac{1}{2}(j-n-1)(2\alpha-2\mu-n-j)}}\ket{f_{j}},
\end{align}
}%
where
{\allowdisplaybreaks
\begin{align}
& V_{n-1,n}^{(d)}=q^{\alpha-n+1}\frac{[n]_q[n-N-1]_q}{a_{n-1}},\\
& V_{n,n}^{(d)}=q^{n-\alpha}\left([{\alpha-\beta}]_q [{\alpha+\beta-N+1}]_q - [\alpha-n+1]_q[\alpha-n]_q\right),\\
& V_{j,n}^{(d)}= q^{j-\alpha}[{\alpha-\beta}]_q [{\alpha+\beta-N+1}]_q a_n a_{n+1}\cdots a_{j-1}\quad (j>n),\\
& V_{n+1,n}^{\top(d*)}=q^{\alpha-n}\frac{[n+1]_q[n-N]_q}{a_{n}},\\
& V_{n,n}^{\top(d*)}=q^{n-\alpha+1}\left([{\alpha-\beta-1}]_q [{\alpha+\beta-N}]_q - [\alpha-n]_q[\alpha-n-1]_q\right),\\
& V_{j,n}^{\top(d*)}= q^{j-\alpha+1}[{\alpha-\beta-1}]_q [{\alpha+\beta-N}]_q a_{j} a_{j+1}\cdots a_{n-1}\quad (j<n),\\ 
&(V^{\top}Z^{\top})_{n+1,n}^{(d*)}=q^{\alpha-N} \frac{\Q{n+1}\Q{N-n}}{a_n}, \label{VZtop1}\\
& (V^{\top}Z^{\top})_{n,n}^{(d*)}=
q^{n-\beta}[\beta]_q[\beta-N+1]_q 
-[n]_q[\alpha-n+1]_q
-q[n-N]_q[\alpha-n-1]_q, \label{VZtop2}\\
& (V^{\top}Z^{\top})_{n-1,n}^{(d*)}=a_{n - 1} q \Q{\alpha - n}\Q{n - \alpha  - 1}, \label{VZtop3}\\
&
(VZ)^{(d)}_{n+1,n} = (V^{\top}Z^{\top})_{n,n+1}^{(d*)},\quad
(VZ)^{(d)}_{n,n} = (V^{\top}Z^{\top})_{n,n}^{(d*)},\quad
(VZ)^{(d)}_{n-1,n} = (V^{\top}Z^{\top})_{n,n-1}^{(d*)}. \label{VZcoe}
\end{align}
}%
Let us recall that the actions of $Z,X,V$ on the vectors $\ket{e_n}$ are given in Subsection \ref{ssec:repe}, and the action of $V$ on the vectors $\ket{f_n}$ is given in \eqref{eq:actVf}. Moreover, actions of the transposed operators on the bases $e^*$ and $f^*$ are deduced using \eqref{eq:transposedmatrixentries}.

\end{document}